\documentclass{amsart}
\usepackage{amsmath, yhmath,amssymb}
\usepackage{bbm}
\usepackage{latexsym}
\usepackage{xcolor}
\usepackage{hyperref}
\usepackage{verbatim} 
\usepackage{mathtools} 
\usepackage[mathcal]{euscript}

\newcommand{\R}{\mathbb{R}}
\newcommand{\C}{\mathbb{C}}
\newcommand{\N}{\mathbb{N}}

\newcommand{\Z}{\mathbb{Z}}
\newcommand{\E}{\mathbb{E}}
\newcommand{\PP}{\mathbb{P}}
\newcommand{\ub}{\mathbf{u}}
\newcommand{\ib}{\mathbf{i}}

\newcommand{\wh}[1]{\widehat{#1}}

\newcommand{\mc}[1]{\mathcal{#1}}
\newcommand{\ov}[1]{\overline{#1}}

\newcommand{\Att}[1]{A_{q_0}^{#1,\tau}}
\newcommand{\Ats}[2]{A_{#2}^{#1,\tau}}

\newcommand{\Lie}{\mathrm{Lie}} 
\newcommand{\Sft}{\mathfrak{S}_\mathcal{F}^\tau}
\newcommand{\Smt}{\mathfrak{S}_S^\tau}

\newcommand{\Prd}{\R\mathbb{P}^{d-1}}
\newcommand{\cl}{\mathrm{cl}}

\newtheorem{thm}{Theorem}
\newtheorem{theorem}[thm]{Theorem}
\newtheorem{lemma}[thm]{Lemma}
\newtheorem{cor}[thm]{Corollary}
\newtheorem{prop}[thm]{Proposition}
\newtheorem{corollary}[thm]{Corollary}
\newtheorem{proposition}[thm]{Proposition}

\newtheorem*{propi}{Proposition}
\newtheorem*{thmi}{Theorem}

\theoremstyle{definition}
\newtheorem{defi}[thm]{Definition}

\theoremstyle{definition}

\theoremstyle{remark} 
\newtheorem{remark}[thm]{Remark}
\newtheorem{example}[]{Example}

\newcommand{\be}{\begin{equation}}
\newcommand{\ee}{\end{equation}}

\newcommand{\GL}{\mathrm{GL}}

\mathtoolsset{showonlyrefs,showmanualtags} 
\numberwithin{equation}{section}

\title[Dwell-time control sets]{
Dwell-time control sets and applications to the stability analysis of 
linear switched systems
}
\date{\today}

\author{Francesco Boarotto}
\address{Dipartimento di Matematica Tullio Levi-Civita,
 Universit\`a degli studi di Padova, Italy}
\email{\href{mailto:francesco.boarotto@math.unipd.it}{\nolinkurl{francesco.boarotto@math.unipd.it}}}

\author{Mario Sigalotti}
\address{Inria Paris \& Laboratoire Jacques-Louis Lions, Sorbonne Universit\'e, Universit\'e Paris-Diderot SPC, CNRS, Inria, 75005 Paris, France}
\email{\href{mailto:Mario.Sigalotti@inria.fr}{\nolinkurl{Mario.Sigalotti@inria.fr}}}

\thanks{The authors have been supported by the ANR SRGI (reference ANR-15-CE40-0018). F.B. is also supported by University of Padova STARS Project
``Sub-Riemannian Geometry and Geometric Measure Theory Issues: Old and
New".
M.S. warmly thanks Fritz Colonius for the exchanges about the results in the paper and for his many useful suggestions. }



\begin{document}
	
\maketitle	
	
\begin{abstract}
We propose an extension of the theory of control sets to the case of inputs satisfying a dwell-time constraint. 
Although the class of such inputs 
 is not closed under concatenation,
we propose a suitably modified definition of control sets that allows to 
recover some important properties known in the concatenable case. 
In particular we apply the control set construction to dwell-time linear switched systems, characterizing their maximal Lyapunov exponent looking only at  trajectories whose  angular component is periodic. 
We also use such a construction to characterize supports of invariant measures for 
random switched systems with dwell-time constraints.
\end{abstract}	
	
	\section{Introduction}

Control sets are geometric objects 
reflecting the structural properties of the family of all reachable sets of a given control system \cite{Coloniusbook}. 
They also happen to be a very helpful  geometric tool for 
investigating the stability properties of linear switched systems. 
For a linear switched system  on $\R^d$ 
with set of modes $S\subset M_d(\R)$, one can consider the 
induced switched system on the projective space $\Prd$ obtained by looking at the angular component of the state. Under the assumption that the set of projected modes, seen as vector fields on $\Prd$, 
satisfies the Lie algebra rank condition, 
there exists a unique nonempty set $D\subset  \Prd$ such that 
the closure of the reachable set from any point in $D$
 is equal to $D$ itself. Moreover, the set $D$, which is called \emph{invariant control set},   has nonempty interior in $\Prd$ and is contained in the closure of the reachable set from any other point in $\Prd$. 
An important application of this property is a useful characterization of maximal and minimal Lyapunov exponents. Under the Lie algebra rank condition
mentioned above, indeed,  these exponents that 
characterize the asymptotic behavior of the system, and in particular its stability, 
 can be computed by looking solely at periodic inputs which generate trajectories whose angular components are also periodic.  
The idea of looking at the Lyapunov exponents associated with the restricted class of periodic trajectories
has been introduced in \cite{ArnoldKlie} in order to relate the asymptotic stability of linear stochastic systems with the behavior of their large deviations, measured in terms of $p$-th mean Lyapunov exponents.  
The characterization in terms of periodic trajectories  turns out to be useful also to prove the continuity of Lyapunov exponents (see  \cite{CK-article} and also \cite{CCS}, where invariant control sets are used to prove continuity for Lyapunov exponents associated with systems subject to persistently exciting inputs).

Another important property of control sets, which is the original motivation for their introduction in the literature \cite{AK1987}, is that they allow to characterize 
supports of invariant measures for piecewise deterministic Markov processes on compact manifolds 
\cite{BenaimColonius}.
(See also \cite{BenaimLeborgne}, where some ad hoc construction of control set
is developed for two-dimensional linear switched systems.) 

A limitation of the control set approach is that it has been developed for the point of view of controllability theory, hence control signals are allowed to vary arbitrarily as time goes. This paper, motivated by the study of switched systems, aims at presenting a first extension of the theory of control sets to 
classes of input signals with restrictions on the switching rule. 
The most studied of these classes consists of signals satisfying \emph{dwell-time constraints}, i.e., the interval between two successive switching times has length larger than a given positive constant. The most important structural difference between the classes of dwell-time and arbitrary switching signals is that the former is not closed under concatenation, in the sense that concatenating two pieces of switching signals satisfying the dwell-time constraint does not necessarily yield a dwell-time signal. Such a lack of concatenability induces the most relevant technical difficulties that we are lead to tackle in this paper.

Our contribution is based on a new suitable notion of control set adapted to the 
dwell-time setting. Instead of basing the construction on reachable sets as in the case of arbitrary switching, we restrict the attention to points which can be attained \emph{at the final time} of a concatenation of constant signals of length larger than the   dwell-time. This leads to a inner estimate of the reachable set, not necessarily connected, but which is shown to be sufficient to yield a proper notion of \emph{dwell-time control set}. 
The latter shares most existence, uniqueness, and topological properties with its counterpart for arbitrary switching. 
(Similar results could have been obtained following the approach in \cite{SanMartin1993}, at least
when the Lie algebra generated by the modes of the switched system is finite-dimensional.) 
 As a consequence, in the linear case, we extend the results of \cite{CK-article} on the characterization of the maximal Lyapunov exponent by trajectories with periodic angular component. We also describe the support of invariant measures for random switched systems with dwell-time.
Most of the paper deals with general systems on a (mostly compact) manifold. The results for linear switched systems are then obtained by considering the induced system on the projective space. 

While proving these results we also partially extend the results known for arbitrarily switching, in particular by removing the 
Lie algebra rank condition
from the characterization of Lyapunov exponents in terms of trajectories with periodic angular components. This is made possible by a new result dealing with the representation of Lie groups, which ensures that the action of the 
group of flows
associated with
 the lifted 
linear switched system has at least one compact orbit on $\mathbb{S}^{d-1}$ (or $\Prd$). This result, which we believe to be of independent interest,  
seems not to be present in the literature.  We are indebted to Uri Bader and Claudio Procesi for their crucial help in providing its proof (presented in the appendix).

The paper is organized as follows. In Section~\ref{s:LMR} we state the main results on linear switched systems contained in the paper.
Section~\ref{s:DTCS} contains the definition and construction of \emph{dwell-time control sets}, establishing their basic properties and illustrating them on a simple example on the projective circle.
In Section~\ref{s:ALSS} dwell-time control systems are specified to the case of linear switched systems and their induced dynamics on the projective space. 
In particular, we show in Theorem~\ref{thm:uniq} that if we restrict a projected linear switched systems to a compact orbit of the corresponding matrix Lie group, then the restricted system admits an unique invariant dwell-time control set (where invariance is defined in a suitable way that takes into account the dwell-time constraint). 
Based on this result, we prove in Theorem~\ref{thm:unifper} that the maximal Lyapunov exponent can be achieved looking only at trajectories with periodic angular component. 
Finally, in Section~\ref{s:stoch} we relate dwell-time control sets and the support of invariant measures for piecewise deterministic dwell-time random processes.

	\section{Linear switched systems: main results}\label{s:LMR}

We present in this section the main results that we obtain for linear switched systems, 
simplifying when necessary the assumptions  and the notations 
with respect to the following sections,  in order to 	
avoid technicalities. 

Given 
an integer $d\ge 1$, a set $S$ of $d\times d$ matrices, and a scalar $\tau\ge 0$, we consider on $\R^d$ the system
\begin{equation}\label{eq:simply}
\dot x(t)=A(t)x(t),
\end{equation}
with $A(\cdot)$ in the set $\mc{S}^\tau$  of all 
piecewise constant functions from $[0,+\infty)$ to  
$S$ such that 
\begin{equation}\label{dwe-t-c}
t_{j+1}(A)-t_j(A)\ge \tau,\qquad \forall j\ge 1,
\end{equation}
 where $(t_j(A))_j$ is the (possibly finite) increasing sequence of discontinuities of $A$.
  A solution to \eqref{eq:simply} is characterized by $A$ and the initial condition $x_0$, and is denoted by $x(\cdot;x_0,A)$.

Notice that the \emph{dwell-time condition} \eqref{dwe-t-c} is equivalent to asking that
$m(\{t_j(A)\}_j)\ge \tau$, where $m(\cdot)$ is the function computing the lower bound of the distance between  
distinct points of a subset of $\R$. 

For every $x_0\in \R^d$, define the \emph{reduced reachable set from $x_0$} 
\[R^\tau(x_0)=\{x(t;x_0,A)\mid t\ge 0,\;A\in \mc{S}^\tau,\;m(\{0\}\cup \{t_j(A)\}_j\cup\{t\})\ge \tau\}.\]
For the case $\tau=0$, the constraint $m(\{0\}\cup \{t_j(A)\}_j\cup\{t\})\ge \tau$ is empty and the set $R^0(x_0)$ is the usual \emph{reachable set from $x_0$} for system \eqref{eq:simply}, in which $A$ is seen as a piecewise constant control with values in $S$.  For $\tau>0$, the constraint $m(\{0\}\cup \{t_j(A)\}_j\cup\{t\})\ge \tau$ does play a role. The reduced reachable set is defined in this way in order to ensure, in particular, the property that
\[R^\tau(x_1)\subset R^\tau(x_0)\quad\mbox{if}\quad x_1\in R^\tau(x_0).\]

	\subsection{Closed orbits for zero-dwell systems}
A technical result that we use several times in the paper is the following. 

\begin{proposition}
Let $d\ge 1$ 
and consider a set $S$ of $d\times d$ matrices
such that $A\in S$ if and only if $-A\in S$. 
Then there exists 
$x_0\in \R^d\setminus\{0\}$ such that the set 
\[\Big\{\frac{x}{\|x\|}\mid x\in R^0(x_0)\Big\}\]
is compact. 
\end{proposition}

The above result can be deduced from the following property  of the actions of linear Lie groups.

	\begin{theorem}\label{thm:closedorbit}
		Let  $B$ be 
		a connected Lie subgroup of $\GL(\R,d)$. 
		Then the action
		\be\label{eq:projaction}
			\varphi:B\times\mathbb{S}^{d-1}\to \mathbb{S}^{d-1}, \quad \varphi(b,x)=\frac{bx}{\|bx\|},
		\ee 
		induced by $B$ on the $(d-1)$-dimensional unit sphere $\mathbb{S}^{d-1}\subset \R^d$, admits at least one closed orbit in $\mathbb{S}^{d-1}$.
	\end{theorem}

The proof of the theorem is based on rather different arguments from those developed in the rest of the paper and 
is postponed to the appendix.

	\subsection{Periodization of trajectories}

Let the set $S$ of $d\times d$ matrices  be bounded. 
The \emph{maximal Lyapunov exponent} for \eqref{eq:simply} is defined as
\begin{equation}\label{eqsup}
\lambda=\sup_{A\in\mc{S}^\tau, x_0\ne 0}\limsup_{t\to+\infty}\frac{\log(\|x(t;x_0,A)\|)}{t}.
\end{equation}

The goal of periodization is to offer a way to approximate $\lambda$ by 
solving finite-horizon maximization problems. 
In order to do so, instead of maximizing   over all trajectories
$x(t;x_0,A)$ with $x_0\ne 0$ and $A\in \mc{S}^\tau$, we might restrict our attention to those for which there exists $T>0$ such that  $A$ is $T$-periodic and $x(T;x_0,A)$ is parallel to $x_0$. 
Let us define 
$\lambda_{\rm per}$ by taking the sup in \eqref{eqsup} restricted to these trajectories (see Section~\ref{sec:periodization} for a detailed definition).  Then clearly 
$\lambda_{\rm per}\le \lambda$. 
 In the case $\tau=0$
Colonius and Kliemann proved in \cite{CK-article} that the two quantities are actually equal, under the assumption that \emph{$S$ 
satisfies the Lie algebra rank condition on $\Prd$},
that is, if the family of vector fields  on the projective space $\Prd$ 
induced by the elements of $S$ 
satisfies the Lie algebra rank condition.
We extend their result by proving equality  for every $\tau\ge 0$ and removing the 
Lie algebra rank condition 
assumption.

\begin{thmi}[Theorem \ref{thm:unifper}]
Let $d
\ge 1$ and consider a bounded set $S$ of $d\times d$ matrices. 
Then $\lambda=\lambda_{\rm per}$.
\end{thmi}

	\subsection{Support of invariant measures}

Consider now the case where $S$ is finite and the switching signal is a random variable.
In order to guarantee that  the dwell time-condition is verified, we assume that the switching times are independent random variables with identical probability distributions having support in $[\tau,+\infty)$. The probability of switching from one element of $S$ to another is encoded by a stochastic matrix. Under these hypotheses, it is possible to associate with the random switching system a probabilistic maximal Lyapunov exponent $\chi_\tau$, which is almost surely attained  by a trajectory of the system. (For details, see Section~\ref{s:stoch} and, in particular, Section~\ref{s:stoch-lin}.)
 
 We then have the following. 
 
 \begin{propi}[Proposition \ref{prop:supportO}]
	Let $S$ be finite and 
	assume that it satisfies the Lie algebra rank condition
on $\Prd$. 
	For every $\theta\in\Prd$, let $R^\tau(\theta)$ be the projection on $\Prd$ of $R^\tau(x)$ for any  vector $x\in \R^d\setminus\{0\}$ projecting to $\theta$. 
		Set $D=\cap_{\theta\in\Prd}\overline{R^\tau(\theta)}$. 
	Then there exists a probability measure $\nu$ on $\Prd\times S$ with 
	\[\mathrm{supp}\nu\subset\cup_{A\in S,\; t\in [0,\tau]}e^{t A}(D)\times \{A\},\] 
	absolutely continuous with respect to the Lebesgue measure, and such that 
	\[
	\chi_\tau=\int_{\Prd\times S}\langle \theta,A \theta\rangle d\nu(\theta,A).
	\]
\end{propi}

	\section{Dwell-time control sets}\label{s:DTCS}
	We introduce  in this section a general construction of control sets in the framework of nonlinear switching systems with a dwell-time constraint. In analogy with \cite{Coloniusbook}, we are going to establish the existence and some crucial geometric properties of these sets that we will extensively use later on in this paper, in particular for studying stability properties of switched systems with dwell-time.

	\subsection{Krener accessibility theorem with dwell-time  constraints}\label{sec:genconst}
	
	Let $M$ be a 
	$d$-dimensional manifold and $U$ any nonempty set. We consider the control system
	\be\label{eq:controlsys}
		\dot{q}(t)=X(q(t),u(t)),\qquad  
		u\in\mc{U},
	\ee
	where  
	$X:M\times U\to TM$ is such that $X(\cdot,u)$ is a smooth complete vector field for every $u\in U$ and 
	\[
	\mc{U}=\{ u:[0,T]\to \R^m\mid T\ge0,\;u(t)\in U\text{ for every } t\in[0,T],\;u\text{ is piecewise constant} \}.
	\]
		Let us denote by $T: \mc{U}\to [0,+\infty)$ the function associating with $u\in   \mc{U}$ the length of the time-interval on which $u$ is defined. 
		We write $\phi(\cdot, q_0,u)$ for the trajectory of \eqref{eq:controlsys} starting at a point $q_0\in M$, and driven by the control $u\in\mc{U}$ on the time-interval $[0,T(u)]$.
		Let, moreover, 
		\[\mc{F}=\{X(\cdot,u)\mid u\in U\}\]
		and denote by $\Lie(\mc{F})$ the 
		Lie algebra of vector fields on $M$ generated by $\mc{F}$ (with respect to the Lie bracket operation).
	
An assumption that we will require for several results is that 
	\be\tag{\textbf{H}}\label{eq:mainhyp}
		\Lie_q(\mc{F})=T_qM,\qquad \text{ for every } q\in M,
	\ee
which is known as the \emph{Lie algebra rank condition}. Here 
$\Lie_q(\mc{F})$ denotes the evaluation of $\Lie(\mc{F})$ at $q$, that is, 
$\Lie_q(\mc{F})=\{Y(q)\mid Y\in \Lie_q(\mc{F})\}$.

	\begin{defi}\label{defi:concatenation}
		Let $u_1,u_2\in \mc{U}$. 
		Then the concatenation $u_1*u_2\in\mc{U}$ is the piecewise constant function defined on $[0,T(u_1)+T(u_2)]$ by
		\[
		(u_1*u_2)(t)=\left\{\begin{array}{ll}
		u_1(t) & \text{on } [0,T(u_1)],\\
		u_2(t-t_1) & \text{on }(T(u_1),T(u_1)+T(u_2)].
		\end{array}
		\right.
		\] 
	\end{defi}
	Definition~\ref{defi:concatenation} immediately extends to an arbitrary finite number of elements of $\mc{U}$.
	
	\begin{defi}
		Let $\tau\ge0$. Let us consider the subset of $\mc{U}$ defined by 
		\[
		\mc{U}^\tau=\left\{ u_1*\dots*u_m \mid m\in\N,\; u_i\in\mc{U}\text{ is constant on }[0,T(u_i)] \text{ and } T(u_i)\geq \tau\mbox{ for } i=1,\dots,m \right\}.
		\]
		In particular, $\mc{U}^0=\mc{U}$. When the parameter $\tau$ is positive, it is commonly referred to as the \emph{dwell-time}. Dwell-time constraints are used to model systems where, for technological, safety, or other reasons,  the actuation of each control value cannot last less than a common bound $\tau$.

		Let us define
		\be
		\Sft=\left\{ \phi(T(u),\cdot,u)\mid u\in\mc{U}^\tau \right\},
		\ee
		that is, $\Sft$ is the subset of diffeomorphisms of $M$ given by the collection of all the 
		flows associated with controls in $\mc{U}^\tau$. Observe that
		$\Sft$ admits the following alternative characterization
		\be\label{eq:semigroup}
		\Sft=\{ e^{t_mX_m}\circ\dots\circ e^{t_1X_1}\mid m\in\N,\; X_1,\dots,X_m\in\mc{F},\;t_1,\dots,t_m\geq\tau\}.
		\ee
	\end{defi}
	Notice that $\Sft$ has the semigroup property. Indeed, for every two elements $g_1,g_2\in\Sft$, their composition $g_1\circ g_2$ is again an element of $\Sft$.

	\begin{defi}[Dwell-time attainable set]
		For every $T>0$, $\tau\ge 0$, and $q_0\in M$, we set
		\be\label{eq:att}
		\Att{T}=\{ \phi(T(u),q_0,u)\mid u\in\mc{U}^\tau,\;T(u)\le T \}.
		\ee
	\end{defi}
	
	Notice that $\Sft{(q_0)}=\bigcup_{T>0}\Att{T}$, where $\Sft{(q_0)}$ stands for the set $\{\psi(q_0)\mid \psi\in \Sft\}$. 
The semigroup property of $\Sft$ implies that $\Sft{(q_0)}$ is dwell-time positively invariant, according to the following definition. 
	\begin{defi}\label{defi:posinv}
		A set $A\subset M$ is said to be \emph{dwell-time positively invariant} if $\Sft(q)\subset A$ for every $q\in A$.
	\end{defi}

	The following is an adaptation of the well-known Krener's theorem to our current situation, that is, in presence of a dwell-time constraint.
	
	\begin{prop}\label{prop:notempty}
	Assume that \eqref{eq:mainhyp} holds true.	Then, for every $T>d\tau$ and $q_0\in M$, $\mathrm{int}(\Att{T})\neq\emptyset$.
	\end{prop}
	\begin{proof}
	Fix $T>d \tau$ and let $\varepsilon=\frac T d-\tau>0$. 
		Since $\Lie_{q_0}(\mc{F})\neq 0$, there exists $X_1\in\mc{F}$ such that $X_1(q_0)\neq 0$. For any $\tau_1>\tau$ we have
		\be
		\frac{d}{dr_1} e^{(\tau_1+r_1)X_1}(q_0)\bigg|_{r_1=0}=e^{\tau_1X_1}_*(X_1(q_0))=
		X_1
		(q_1)\neq 0,
		\ee
		where we set $q_1=e^{\tau_1X_1}(q_0)$ for some  $\tau_1\in(\tau,\tau+\varepsilon)$ to be fixed later. Next, if $d\geq 2$, there exist $\tau_1\in(\tau,\tau+\varepsilon)$ and $X_2\in\mc{F}$ such that 
		$X_1(q_1)\wedge X_2(q_1)\neq 0$, for otherwise $\Lie_{q_1}(\mc{F})$ would be of dimension one, and \eqref{eq:mainhyp} would be violated. Notice moreover that $X_1(q_1)\wedge X_2(q_1)\neq 0$
		is still true if we vary $\tau_1$ in a small neighborhood. 
		Then, for any $\tau_2\in(\tau,\tau+\varepsilon)$, we compute
		\begin{align}
		\frac{d}{dr_1}e^{(\tau_2+r_2)X_2}\circ e^{(\tau_1+r_1)X_1}(q_0)\bigg|_{(r_1,r_2)=0}&
		=e^{\tau_2X_2}_*(X_1(q_1))=(e^{\tau_2 X_2}_*X_1)(e^{\tau_2X_2}(q_1)),\\
		\frac{d}{dr_2}e^{(\tau_2+r_2)X_2}\circ e^{(\tau_1+r_1)X_1}(q_0)\bigg|_{(r_1,r_2)=0}&
		=e^{\tau_2X_2}_*(X_2(q_1))=X_2(e^{\tau_2X_2}(q_1)).
		\end{align}
		The two tangent vectors at $e^{\tau_2X_2}\circ e^{\tau_1X_1}(q_0)$ are linearly independent, since $e^{\tau_2X_2}_*$ is a diffeomorphism of the tangent space. 
		Again, if $d\geq 3$, we set $q_2=e^{\tau_2X_2}\circ e^{\tau_1X_1}(q_0)$, and, up to eventually modifying $\tau_1,\tau_2\in(\tau,\tau+\varepsilon)$, we can find $X_3\in\mc{F}$ 
		such that 
		\[
		(e^{\tau_2X_2}_*X_1)(q_2)\wedge X_2(q_2)\wedge X_3(q_2)\neq 0,
		\]
		for otherwise $\Lie_{q_2}(\mc{F})$ would be contained in 
		the tangent plane to the surface parameterized by
		$(r_1,r_2)\mapsto e^{(\tau_2+r_2)X_2}\circ e^{(\tau_1+r_1)X_1}(q_0)$ (with $(r_1,r_2)$ in a neighborhood of $(0,0)$) and 
		\eqref{eq:mainhyp} would not hold. 
		Repeating the argument above, suppose that $d\geq k$ and that at the $k$-th step we define 
		\[
			q_{k-1}= e^{\tau_{k-1}X_{k-1}}\circ \dots\circ e^{\tau_1 X_1}(q_0).
		\]
		Then, up to slightly modifying $\tau_1,\dots,\tau_{k-1}\in (\tau,\tau+\varepsilon)$,  we can find $X_{k}\in\mc{F}$ such that
		\[
		X_k(q_{k-1})\wedge \dots\wedge ( (e^{\tau_{k-1}X_{k-1}}\circ\dots\circ e^{\tau_1X_1} )_*X_1 )(q_{k-1})\neq 0.
		\]
		Differentiating at zero the map
		\be\label{eq:maptodiff}
		(r_1,\dots,r_k)\mapsto e^{(\tau_k+r_k)X_k}\circ\dots\circ e^{(\tau_1+r_1)X_1}(q_0),
		\ee
		we find that
		\begin{align}\label{eq:family}
		\\\frac{d}{dr_i} e^{\tau_kX_k}\circ\dots\circ e^{(\tau_i+r_i)X_i} \circ\dots\circ e^{\tau_1X_1} (q_0)\bigg|_{r_i=0}&=e^{\tau_kX_k}_*( (e^{\tau_{k-1}X_{k-1}}\circ\dots\circ e^{\tau_i X_i})_*X_i )(q_{k-1})\\&=( (e^{\tau_{k}X_{k}}\circ\dots\circ e^{\tau_{i+1} X_{i+1}})_*X_i )(q_{k})
		\end{align}
		for $i=1,\dots,k$.
		The differential of \eqref{eq:maptodiff} is then of maximal rank locally at $(0,\dots,0)$ and the claim is proved after $d$ steps.
	\end{proof}

\begin{remark}\label{rmk:bestkren}
The proof of Proposition~\ref{prop:notempty} actually shows that if \eqref{eq:mainhyp} holds true, 	then, for every $\delta>0$ and $q_0\in M$ there exist $\tau_1,\dots,\tau_d\in (\tau,\tau+\delta)$ and $X_1,\dots,X_d\in \mathcal{F}$ such that $\R^d\ni (t_1,\dots,t_d)\mapsto  e^{t_dX_d}\circ\dots\circ e^{t_1X_1}(q_0)$ 
is  a local diffeomorphism at $(\tau_1,\dots,\tau_d)$.
\end{remark}
	
	\begin{proposition}\label{prop:ReverseKren}
		Let $q_0\in M$ and 
		\[
			g:\R^d\to M,\qquad g:(t_1,\dots,t_d)\mapsto  e^{t_dX_d}\circ\dots\circ e^{t_1X_1}(q_0), 
		\] 
		be a local diffeomorphism at $(\tau_1,\dots,\tau_d)$. Let $q_1=g(\tau_1,\dots,\tau_d)(q_0)$. Then 
		\[
			h:\R^d\to M,\qquad h:(t_1,\dots,t_d)\mapsto e^{-t_1X_1}\circ\dots\circ e^{-t_kX_k}(q_1),
		\] 
		is also a local diffeomorphism at $(\tau_1,\dots,\tau_d)$.
		
		In particular, there exist an arbitrarily small neighborhood $N$ of $(\tau_1,\dots,\tau_d)$ in $\R^d$ and 
		$\Omega_0$, $\Omega_1$ neighborhoods of $q_0$, $q_1$, respectively, such that for every $p_0\in\Omega_0$ and $p_1\in\Omega_1$ there exists $(t_1,\dots,t_d)\in N$ such that $p_1=g(t_1,\dots,t_d)(p_0)$.
	\end{proposition}
	
	\begin{proof}
		By assumption (compare with \eqref{eq:family}), the vector fields $X_1,\dots, X_d$ appearing in the definition of $g$ satisfy the relation
		\be\label{eq:relation}
		\bigwedge_{i=1}^d \left( (e^{\tau_d X_d}\circ\dots\circ e^{\tau_{i+1} X_{i+1}})_* X_i \right)(q_1)\neq 0.
		\ee
		On the other hand, differentiating at $0$ with respect to $r_i$, $i=1,\dots,d$, the map
		\[
		(r_1,\dots,r_d)\mapsto   e^{-(\tau_1+r_1) X_1}\circ \dots\circ e^{-(\tau_d+r_d)X_d}(q_1),
		\]
		we obtain
		\begin{align}
		\frac{d}{dr_i}e^{-\tau_1X_1}  \circ\dots\circ e^{-(\tau_i+r_i)X_i} \circ\dots\circ  e^{-\tau_dX_d} (q_1)\bigg|_{r_i=0}&=-((e^{-\tau_1X_1}\circ\dots\circ e^{-\tau_{i-1}X_{i-1}})_*X_i)(q_0)\nonumber\\
		&=-((e^{-\tau_1X_1}\circ\dots\circ e^{-\tau_{i}X_{i}})_*X_i)(q_0).\label{eq:relation2}
		\end{align}
		To see that the collection of vectors obtained in \eqref{eq:relation2} is a basis of $T_{q_0}M$, it is sufficient to observe that they can be mapped to the collection appearing in \eqref{eq:relation} applying the push-forward $(e^{\tau_dX_d}\circ\dots\circ e^{\tau_1X_1})_*$. Since this latter family is a basis of $T_{q_1}M$, then $h$ is a local diffeomorphism at $(\tau_1,\dots,\tau_d)$.
		
		The last part of the statement follows by a standard compactness/continuity argument. 
	\end{proof}

	Combining Remark~\ref{rmk:bestkren} and Proposition~\ref{prop:ReverseKren}  
	yields the following. 
	\begin{corollary}\label{cor:revers}
	Assume that \eqref{eq:mainhyp} holds true.	Then, for every $T>d\tau$ and $q_0\in M$, 
	there exist $u\in \mc{U}^\tau$ with $T(u)<T$ and a neighborhood $\Omega$ of $q_0$ in $M$
	such that for every $q_1\in \Omega$ there exists $w\in   \mc{U}^\tau$ with $T(w)<T$ and
	$\phi(T(w),q_1,w)=\phi(T(u),q_0,u)$. 
		\end{corollary}
	
	\subsection{Control sets} 	Let us denote by $\mc{U}^\tau_\infty$ the set of all piecewise constant functions $u:[0,+\infty)\to U$  
	such that there exists $T_n\rightarrow +\infty$ for which  $u|_{[0,T_n]}\in \mc{U}^\tau$ for every $n\in \N$.
	With a slight abuse of notation, given $u:[0,T]\to U$ (respectively, $u:[0,+\infty)\to U$) and $0\le t_1\le t_2\le T$ (respectively, $0\le t_1<+\infty$), we say that $u|_{[t_1,t_2]}$ is in  $\mc{U}^\tau$ (respectively, $u|_{[t_1,+\infty)}$ is in  $\mc{U}^\tau_\infty$) if $[0,t_2-t_1]\ni s\mapsto u(s+t_1)$ (respectively, $[0,+\infty)\ni s\mapsto u(s+t_1)$) does. 
	
A special role in what follows is played by those times in $[0,+\infty)$ at which  
 a 
signal
$u\in  \mc{U}^\tau_\infty$ is split into two 
subsignals which both satisfy  the dwell-time condition. Such times are interesting because they can be used to modify $u$ by concatenating one of the subsignals with any other dwell-time signal. Such partial concatenability allows to prove the interesting properties that we seek to bring to light in control sets.

	\begin{defi}\label{defi:DTCS}
		A 
		subset $D\subset M$ is said to be a \emph{$\tau$ dwell-time control set} ($\tau$-CS, for short) for system~\eqref{eq:controlsys} if it satisfies the following conditions:
		\begin{itemize}
			\item [(i)] For every $q\in D$ there exists $u\in \mc{U}^\tau_\infty$ such that $\phi(t,q,u)\in D$ for all $t>0$ for which both $u|_{[0,t]}$ is in $\mc{U}^\tau$ and $u|_{[t,+\infty)}$ is in $\mc{U}^\tau_\infty$;
			\item [(ii)] For every $q\in D$, one has that $D\subset \ov{\Sft(q)}$;
			\item [(iii)] $D$ is maximal with these properties, i.e., if $D'\supset D$ satisfies both (i) and (ii), then $D'=D$.
		\end{itemize}
		Moreover, a $\tau$-CS $D\subset M$ is said to be an \emph{invariant $\tau$-CS} ($\tau$-ICS, for short) if $\ov{D}=\ov{\Sft(q)}$ for every $q\in D$.
	\end{defi}
	
	\begin{remark}\label{rem:closedinvset}
		Observe that, if $D$ is a subset of $M$ such that 
\begin{equation}\label{eq:tICSclosed}
D=\ov{\Sft(q)}\qquad \forall q\in D,
\end{equation}
 then automatically $D$ is a $\tau$-ICS.
	Moreover, if 	\eqref{eq:mainhyp}  holds true, then any $\tau$-ICS is closed.
	Indeed, 
	consider a $\tau$-ICS $C$ and let us prove that $D=\ov{C}$ satisfies \eqref{eq:tICSclosed}. The conclusion then follows by maximality. 
It is clear that for every $q\in \ov{C}$, the set $\Sft(q)$ (and hence $\ov{\Sft(q)}$) is contained in $\ov{C}$. 
		In order to conclude, fix $q\in \ov{C}$ and  let us notice that it is enough to show that $\Sft(q)$ intersects $C$, since for every $p\in \Sft(q)\cap C$ we have $\ov{C}=\ov{\Sft(p)}\subset \ov{\Sft(q)}$. 
		Thanks to Remark~\ref{rmk:bestkren} and Proposition~\ref{prop:ReverseKren}, there 
		exists $u\in \mc{U}^\tau$ such that $\Sft(q)$ covers a neighborhood of 
		$\phi(T(u),q,u)$. Since, moreover, $\phi(T(u),q,u)$ is in $\ov{C}$, we indeed have that $\Sft(q)$ intersects $C$. 
	\end{remark}
	
	\begin{remark}\label{rem:countable}
		Observe that 
		if \eqref{eq:mainhyp} holds true then $M$ contains at most countably many distinct $\tau$-ICS. Indeed, $M$ admits a countable dense subset $\{q_n\}_{n\in\N}$. Let $D$ be any $\tau$-ICS. Then $D$ is characterized by \eqref{eq:tICSclosed} according to Remark~\ref{rem:closedinvset}. 
				By Proposition~\ref{prop:notempty}, $\mathrm{int}(D)\neq \emptyset$, hence there exists $N\in\N$ such that $q_N$ is in $\mathrm{int}(D)$. But then $D=\ov{\Sft(q_N)}$. The family of all $\tau$-ICS is therefore contained in 
		$\{\ov{\Sft(q_N)}\mid N\in\N\}$.
	\end{remark}

Our first result on control sets is an existence property for invariant dwell-time control sets when $M$ is compact.

	\begin{theorem}\label{thm:existence}
		Let $M$ be compact. For each $q\in M$ there exists a nonempty $\tau$-ICS $D_q$ contained in $\ov{ \Sft(q) }$.
		If, moreover, \eqref{eq:mainhyp} holds true, then $D_q$ has nonempty interior. 
	\end{theorem}
	
	\begin{proof}
		Let $q\in M$. Consider the collection $\mathfrak{V}_q=\{ \ov{ \Sft(p) }\mid p\in \ov{ \Sft(q) } \}$. Then $\mathfrak{V}_q$ is nonempty, and all  elements of $\mathfrak{V}_q$ are dwell-time positively invariant, 
	since 
		\begin{equation}\label{e:new}
		\Sft(z)\subset\ov{ \Sft(p) } \qquad \forall z\in\ov{ \Sft(p) }.
		\end{equation}
		To check this, we let $w=g(z)$, for some $g\in \Sft$. We need to prove that $w\in\ov{ \Sft(p) }$. 
		Let $O_w$ be any neighborhood of $w$. Since $z\in \ov{ \Sft(p) }$, then, for any neighborhood $V_z$ of $z$, $V_z\cap \Sft(p)\neq\emptyset$. We choose $V_z$ so small that $g(V_z)\subset O_w$, and then for any $y\in V_z\cap \Sft(p)$, we have that $g(y)\in O_w\cap \Sft(p)$, since $g(\Sft(p))\subset \Sft(p)$.
		
		Now observe that $\mathfrak{V}_q$ is a partially ordered (with respect to the inclusion) collection of nonempty compact sets. The Cantor intersection theorem implies then that every descending chain 
		$\{ \ov{\Sft(q_i)}\mid i\in I \}$ is such that $\bigcap_{i\in I} \ov{\Sft(q_i)}\neq \emptyset$. Therefore, 
		thanks again to \eqref{e:new}, 
		every chain has a lower bound of the form $\ov{\Sft(p)}$, for some $p\in\bigcap_{i\in I}\ov{\Sft(q_i)}$. 
		
		The dual form of Zorn's lemma applies, and yields that $\mathfrak{V}_q$ has a minimal element $D_q$ of the form $\ov{\Sft(p)}$, for some $p\in\ov{\Sft(q)}$. 
		In particular,  $D_q$ is dwell-time positively invariant.
		 By Remark~\ref{rem:closedinvset}, in order to prove that $D_q$ is a $\tau$-ICS it remains to show that $D_q=\ov{\Sft(y)}$ for all $y\in D_q$. Since $D_q$ is closed and dwell-time positively invariant, then $\ov{\Sft(y)}\subset D_q$. Equality then follows from the minimality property of $D_q$. 
		
		Finally, Proposition~\ref{prop:notempty} implies that $D_q$ has nonempty interior whenever \eqref{eq:mainhyp} holds true. 
	\end{proof}

	\subsubsection{Properties of  dwell-time control sets}
	
	\begin{lemma}\label{lemma:nonexit}
		Let $D$ be a 
		subset of $M$ which is maximal with respect to property {\rm (ii)} in Definition~\ref{defi:DTCS}.
		Let $q\in D$ and $u\in\mc{U}^\tau$ be such that $\phi(T(u),q,u)\in D$. Then $\phi(t,q,u)\in D$ for all $0\leq t\leq T(u)$ such that both $u|_{[0,t]}\in\mc{U}^\tau$ and $u|_{[t,T(u)]}\in\mc{U}^\tau$.
	\end{lemma}
	
	\begin{proof}
		Let us write $T$ for $T(u)$. It suffices to prove that for every $0\leq t\leq T$ such that both $u|_{[0,t]}\in\mc{U}^\tau$ and $u|_{[t,T]}\in\mc{U}^\tau$, and every $p\in D$, both $p\in\ov{ \Sft( \phi(t,q,u) ) }$ and $\phi(t,q,u)\in\ov{ \Sft(p) }$. Indeed if this were true and $\phi(t,q,u)\not\in D$, then $D\cup \{\phi(t,q,u)\} $ would be a set strictly larger than $D$ that satisfies 
		(ii) in Definition~\ref{defi:DTCS}.

		To check the second assertion, fix a neighborhood $V$ of $\phi(t,q,u)$. Then there exists a neighborhood $U$ of $q$ such that $\phi(t,z,u)$ in $V$ for all $z\in U$. Since $q\in D$, for every $p\in D$ we have that $q\in\ov{ \Sft(p) }$, and then there exists $u_1\in\mc{U}^\tau$ with $\phi(T(u_1),p,u_1)\in U$. We concatenate $u_1$ and $u|_{[0,t]}$ on $[0,T(u_1)+t]$ to see that $u_1*u\in\mc{U}^\tau$ and $\phi(T(u_1)+t,p,u_1*u)\in V$, so we are done. The other assertion is easier, for it suffices to observe that $p\in\ov{ \Sft
(\phi(T,q,u)) }\subset \ov{ 
\Sft(\phi(t,q,u)) }$, as $u|_{[t,T]}\in\mc{U}^\tau$ by assumption.
	\end{proof}
	
	\begin{lemma}\label{lemma:infinite}
		Let $D$ be a subset of $M$ that is maximal with respect to property 
		{\rm (ii)} in Definition~\ref{defi:DTCS}. Assume that $\mathrm{int}(D)\neq \emptyset$. Then $D$ is a $\tau$-CS.
	\end{lemma}
	
	\begin{proof}
		It is sufficient to show that, for every $q\in D$, there exists $u\in\mc{U}^\tau_\infty$ such that
		$\phi(t,u,q)\in D$ for all $t\ge 0$ such that $u|_{[0,t]}$ is in $\mc{U}^\tau$ and $u|_{[t,+\infty)}$ in $\mc{U}^\tau_\infty$. 
		
		Pick $q\in D$ and fix two nonempty open subsets $V,W$ of $\mathrm{int}(D)$ such that $V\cap W=\emptyset$.
		Applying twice property 
		(ii) in Definition~\ref{defi:DTCS}	,  
		there exist $u_0,u_1\in\mc{U}^\tau$,
		$q_1\in V$, and  $q_2\in W$ such that $q_1=\phi(T(u_0),q,u_0)$ and $q_2=\phi(T(u_1),q_1,u_1)$. 
		Up to restricting $V$ and $W$, we can assume that $W=\phi(T(u_1),V,u_1)$.
			Applying iteratively property~(ii) in Definition~\ref{defi:DTCS}, we select a sequence  $\{q_i\}_{i\in \N}$ such that $q_i\in V$ for $i$ odd and $q_i\in W$ for $i$ even, and a sequence $\{u_i\}_{i\in\N}$ in $\mc{U}^\tau$ such that $q_{i+1}=\phi(T(u_i),q_i,u_i)$ for every $i\in\N$ and $u_i=u_1$ for $i$ odd. Notice that $\sum_{i\in\N}T(u_i)\ge \sum_{i\in\N}T(u_{2i+1})=\sum_{i\in\N}T(u_{1})=\infty$. Hence, the concatenation of all $u_i$ is an element of  $\mc{U}^\tau_\infty$, which we denote by $u$.

		Finally, if $t\ge 0$ is such that both $u|_{[0,t]}$ is in $\mc{U}^\tau$ and $u|_{[t,\infty)}$ in $\mc{U}^\tau_\infty$,
		then let $i\in \N$ be such that $t\le T_i$ 	and notice that $u|_{[t,T_{i+1}]}$ is in $\mc{U}^\tau$. Applying Lemma~\ref{lemma:nonexit} 
		with $u=u|_{[0,T_{i+1}]}$, we deduce that 
		$\phi(t,q,u)$ is in $D$. Then $D$ is a $\tau$-CS.
	\end{proof}

	\begin{lemma}\label{lemma:properties}
		Let \eqref{eq:mainhyp} hold true and 	assume that $D\subset M$ is a $\tau$-ICS. Then
		\begin{itemize}
			\item [(i)] $\ov{\mathrm{int}(D)}=D$;
			\item [(ii)] $\mathrm{int}(D)$ is dwell-time positively invariant;
			\item [(iii)] There exists an open and dense subset $C$ of $D$  such that $C\subset \bigcap_{q\in D} \Sft{(q)}$;
			\item [(iv)] There exists an open and dense subset $\mathfrak{C}$ of $D$ such that $\mathfrak{C}=\Sft{(q)}$ for all $q\in \mathfrak{C}$.
			
		\end{itemize}
	\end{lemma}
	
	\begin{proof} 
		Let us prove (i). Recall that $\mathrm{int}(D)\neq \emptyset$ by Proposition~\ref{prop:notempty}. The  inclusion of $\ov{\mathrm{int}(D)}$ in $D$ is obvious from Remark~\ref{rem:closedinvset}. As for the reverse inclusion, let us choose  $q\in D$ and any neighborhood  $V_q$ of $q$. We need to show that 
		\begin{equation}\label{eq:renew}
		V_q\cap\mathrm{int}(D)\neq \emptyset. 
		\end{equation}
		Using Proposition~\ref{prop:notempty} we first deduce that there exists a nonempty open set $O$ contained in $\Sft(q)$. 
		Since $\Sft(q)\subset D$, we have that $O$ is contained in $\mathrm{int}(D)$. 
		Let $p\in O$. By definition of $\tau$-ICS, $V_q\cap \Sft(p)\neq \emptyset$, and we pick $q'\in V_q\cap\Sft(p)$. Hence $q'=g(p)$, for some $g\in\Sft$.  Since $V_q$ is open, we can find a sufficiently small neighborhood $V_p\subset O$ of $p$ such that $V_{q'}:=g(V_p)$ is a neighborhood of $q'$ entirely contained in $V_q$. Then $V_{q'}$ is a nonempty open set in $\Sft(q)$, in particular it is contained in $\mathrm{int}(D)$, and \eqref{eq:renew} is proved.

		As for (ii) we proceed as follows.  Let $g\in\Sft$ 
		and observe that since $D$ is dwell-time positively invariant, then $g(\mathrm{int}(D))$ is contained in $D$. Since, moreover, $g$ is a diffeomorphism, then $g(\mathrm{int}(D))$ is open and hence contained in $\mathrm{int}(D)$.

		Consider now point (iii). Let $V$ be any nonempty open set in $D$. 
		We should prove that 	$\bigcap_{q\in D} \Sft{(q)}\cap V$ has nonempty interior. 
		Fix some $q_0\in D$. According  to Remark~\ref{rmk:bestkren},
there exist
		$\tau_1,\dots,\tau_d>\tau$ and $X_1,\dots,X_d\in\mc{F}$ such that 
		$(t_1,\dots,t_d)\mapsto e^{t_dX_d}\circ\dots\circ e^{t_1 X_1}(q_0)$ is a local diffeomorphism at $(\tau_1,\dots,\tau_d)$. 
		Notice that
		\[
			q_1:=e^{\tau_dX_d}\circ\dots\circ e^{\tau_1 X_1}(q_0) 
		\]
		is in $D$ 
		and select $k\in \Sft$ such that $k(q_1)\in V$.
		Thanks to Proposition~\ref{prop:ReverseKren},
		there exist $\Omega_0,\Omega_1$, neighborhoods of $q_0,q_1$ respectively, such that for every $p_0\in\Omega_0$ and $p_1\in\Omega_1$ there exists $t_1,\dots,t_d>\tau$ such that
		\[
			p_1=e^{t_dX_d}\circ\dots\circ e^{t_1 X_1}(p_0).
		\]
		Up to reducing $\Omega_0$ and $\Omega_1$, we can assume that
		$k(\Omega_1)\subset V$.
		We now complete the proof of (iii) by showing that $k(\Omega_1)$ is also  contained in $\bigcap_{q\in D} \Sft{(q)}$. 		  
		Indeed, for any $q\in D$ and $p_1\in \Omega_1$ there exist $g\in \Sft$ and $t_1,\dots,t_d>\tau$ such that $g(q)\in \Omega_0$ and
		\[
			p_1=e^{t_d X_d}\circ\dots\circ e^{t_1X_1}\circ g(q).
		\]
		Hence, $k\circ e^{t_d X_d}\circ\dots\circ e^{t_1X_1}\circ g$ is in $\Sft$ and maps 
		$q$ to $k(p_1)$.
		
		Let us finally prove (iv). Let $\mathfrak{C}$ be the union of all sets $C$ as in (iii). 
		Then $\mathfrak{C}$ is dwell-time positively invariant,  
		since for every $C$ as in (iii) and every $g\in \Sft$,  $C\cup g(C)$ is also as in (iii).
		Hence, for $q\in \mathfrak{C}$, $\Sft{(q)}\subset \mathfrak{C}$ which in turn  
		is contained in $\bigcap_{q'\in D} \Sft{(q')}\subset \Sft{(q)}$ because of (iii). We conclude that $\Sft{(q)}= \mathfrak{C}$, as required. 
	\end{proof}
	
Applying the previous results to the case where $M$ is compact, we can strengthen the conclusions of Remark~\ref{rem:countable}	
under the compactness assumption.
	
	\begin{prop}\label{prop:finite}
		Let $M$ be a compact manifold and assume that  \eqref{eq:mainhyp} holds true. Then there exist finitely many distinct $\tau$-ICS on $M$.
	\end{prop}
	
	\begin{proof}
Reasoning by contradiction and	 following Remark~\ref{rem:countable}, let $\{D_n\}_{n\in\N}$ be the countable collection of all distinct $\tau$-ICS on $M$. 
		For each $D_n$, consider the open and dense subset $\mathfrak{C}_n\subset D_n$ constructed as in point (iv) of Lemma~\ref{lemma:properties}. For every $n\in\N$ let us choose an element $q_n\in \mathfrak{C}_n$. Up to selecting a subsequence, still denoted by $\{q_n\}_{n\in\N}$, we may assume that there exists $\lim_{n\to\infty}q_n=:q\in M$.
		
		Using Theorem~\ref{thm:existence}, let us consider a $\tau$-ICS $D_q\subset\ov{\Sft(q)}$ and, accordingly, the open and dense set $\mathfrak{C}_q$ contained in $D_q$. Then there exist $p\in \mathfrak{C}_q$ and $g\in\Sft$ such that $p=g(q)$. Fix $U_p\subset \mathfrak{C}_q$ and $V_q\subset M$ neighborhoods  of $p$ and  $q$ respectively, such that $g(V_q)\subset U_p$. For $n$ large enough, $q_n\in V_q$. Then  
		$p_n:=g(q_n)$ is an element of $\mathfrak{C}_q$ for $n$ large enough.
		
		In particular, for every such $n$, the inclusion
		\be\label{eq:incl}
		\mathfrak{C}_q=\Sft{(p_n)}\subset \Sft{(q_n)}=\mathfrak{C}_n
		\ee
		holds true. 
		Taking the closures on both sides of \eqref{eq:incl}, we deduce that $D_q\subset D_n$. The maximality property of $D_q$ (point (iii) in Definition~\ref{defi:DTCS}) then implies that $D_q=D_n$ for every $n$ large enough, leading to  a contradiction.
	\end{proof}
	
	\begin{example}\label{ex:first}
		Let us consider, on the one-dimensional projective space $\R\mathbb{P}^1$, two vector fields $f_1,f_2$ as in Figure~\ref{fig:first}. 
		
			\begin{figure}[h!]
			\includegraphics[width=11cm]{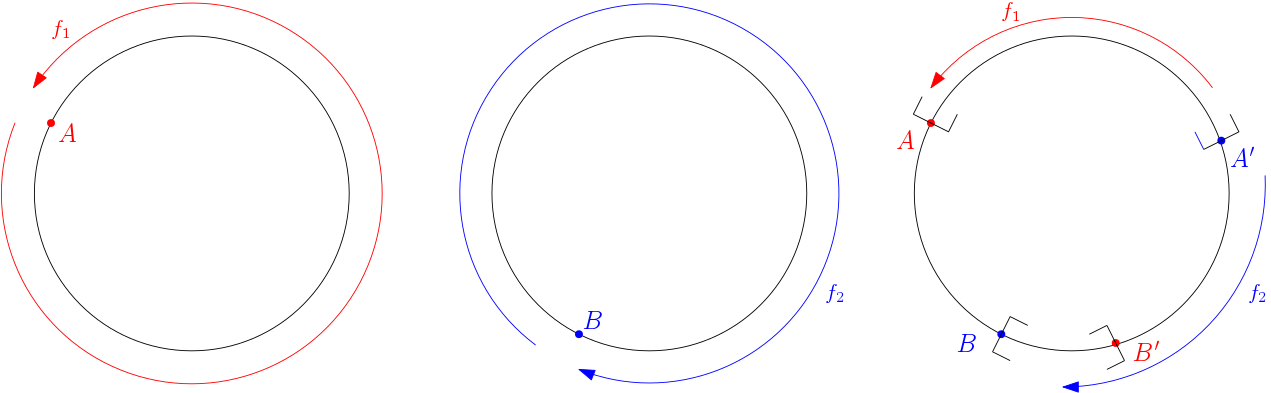}
			\caption{}	
			\label{fig:first}
		\end{figure}
	
		The vector field $f_1$ has a single equilibrium $A$ and its trajectories outside $A$ turn counterclockwise, while  $f_2$ has a single equilibrium $B$ different form $A$ and its trajectories outside $B$ turn clockwise. Each of the vector fields $f_1$ and $f_2$ can be obtained, for instance, by projecting on $\R\mathbb{P}^1$ a linear  dynamics conjugate to the one of the double integrator
		\[\dot x=\begin{pmatrix}0&1\\0&0\end{pmatrix} x,\qquad x\in \R^2.\]
		
		Fix a dwell-time $\tau>0$ and consider the family $\mathcal{F}=\{f_1,f_2\}$. Define the points $A'$ and $B'$ by the relations
		\[
		A'=e^{\tau f_1}(B),\quad B'=e^{\tau f_2}(A).
		\] 
		In what follows, the symbol $\wideparen{XY}$  denotes the closed arc 
		connecting $X$ to $Y$ in the clockwise direction.
 We claim that there exists a unique $\tau$-ICS $D_\tau$ associated with $\mathcal{F}$, given by the union $\wideparen{AA'}\cup \wideparen{B'B}$. In order to check it, first notice that $\wideparen{AA'}$ is positively invariant for $f_1$, $\wideparen{B'B}$ is positively invariant for $f_2$, and 
		\[
			e^{tf_2}(\wideparen{AA'})\subset \wideparen{B'B},\quad e^{tf_1}(\wideparen{B'B})\subset \wideparen{AA'},\qquad \forall t\ge \tau.
		\]
		Let us show now that, for every $B''\in\wideparen{B'B}$, we have $D_\tau
		\subset \ov{\Sft{(B'')}}$. To see this notice that $ e^{t f_1}(B'')\to A$ as $t\to +\infty$.
		Then we can reach from $B''$ any point in the interior of $\wideparen{B'B}$ (and in particular we can get as close as we want to the two boundary points) by taking $t$ large and then applying $f_2$ during a time $t'$  larger than $\tau$. Taking $t'$ large and applying again $f_1$ for a time larger than $\tau$, we can reach any point in the interior of $\wideparen{AA'}$, proving that $\ov{\Sft(B')}$ contains $D_\tau$.
		The symmetric  argument starting from some point $A''$ of $\wideparen{AA'}$ concludes the proof of the fact that  $D_\tau$ 
		is a $\tau$-ICS. The uniqueness of $D_\tau$ follows from the fact that, by dwell-time invariance,  any $\tau$-ICS must contain the global attractive equilibria for $f_1$ and $f_2$, that is, $A$ and $B$.
		
		Notice that $D_\tau$  
		is connected (and actually equal to $D_0=\wideparen{AB}$) for $\tau$ small and not connected for $\tau$ large. Consider the critical situation, that is, the value of $\tau$ for which  $A'= B'$. In this case $D_\tau$ coincides with the arc $\wideparen{AB}$. We claim that, in this case, the set $\mathfrak{C}$ of point (iv) in Lemma~\ref{lemma:properties} is $\mathrm{int}(\wideparen{AB})\setminus\{ A'\}$. Indeed, $A'$ can be reached only by starting either at $A$ or $B$. Since these points are equilibria, they cannot be reached using the semigroup $\Sft$ as soon as we start within $\mathrm{int}(\wideparen{AB})$.
		This example shows, in particular, that $\mathfrak{C}$ may be different from the interior of the corresponding $\tau$-ICS,  
		in contrast with the case $\tau=0$ (see \cite[Theorem 3.1.5]{Coloniusbook}). 
	\end{example}

	\section{Applications to linear switched systems}  \label{s:ALSS}
	
	\subsection{The maximal Lyapunov exponent of a linear switched system}
	 Let $S\subset M_d(\R)$ be a bounded set of matrices, playing the role of the control set $U$ in the previous section, and for $\tau\ge 0$ denote by $\mc{S}^\tau$ and $\mc{S}^\tau_\infty$ the sets of piecewise constant signals with values in $S$ and dwell-time $\tau$, in analogy with the sets $\mc{U}^\tau$ and $\mc{U}^\tau_\infty$ introduced earlier. 
We specify the notation to the set $S$ in order to stress that the vector fields considered here are linear and that they are identified with the corresponding matrices. 
	The associated system is 
	\be\label{eq:lss}\tag{$\Sigma_\tau$}
		\dot{x}(t)=A(t)x(t), \qquad t\geq 0, \quad A(\cdot)\in\mc{S}^\tau_\infty, \quad x\in\R^d.
	\ee
	We call \eqref{eq:lss} a \emph{switched system}, in order to stress that we focus on 
	the uniform asymptotic properties of its dynamics, looking at 
	$A(\cdot)$ as a \emph{switching signal} 
	rather than as a control law. 
	Following the usual switched systems terminology, we refer to each element of $S$ as a \emph{mode} of \eqref{eq:lss}. 

For every  $A\in \mc{S}^\tau_\infty$ 
we denote by 
$\Phi_A(\cdot,\cdot):[0,+\infty)\times [0,+\infty)\to \GL(\R,d)$ the \emph{fundamental matrix} of   \eqref{eq:lss}, that is, the solution to
\[ \frac{d}{dt}\Phi_A(t,t_1)=A(z)\Phi_A(t,t_1),\quad \Phi_A(t_1,t_1)=\mathrm{Id}_d.\]
	It is immediate to see that  $t\mapsto \Phi_A(t,t_1)x_0$ is the unique solution to \eqref{eq:lss} with $x(t_1)=x_0$ and that $\Phi_A(t_1,t)=\Phi_A(t,t_1)^{-1}$.

	The boundedness of the set ${S}$ ensures that the trajectories of system~\eqref{eq:lss} have at most exponential growth with a common upper bound on their growth rates. 
The lower of such upper bounds is the object of  the following definition.
	
	\begin{defi}\label{defi:unifexprate}
		Given $\tau\ge 0$ and a bounded set $S\subset M_d(\R)$, the \emph{uniform exponential rate} $\lambda_\tau(S)$ and the maximal Lyapunov exponent $\wh{\lambda}_\tau(S)$ are defined, respectively, by 
		\be\label{eq:unifexprate}
			\lambda_\tau(S)=\limsup_{t\to+\infty}\sup_{A\in\mc{S}^\tau_\infty}\frac{\log(\|\Phi_A(t,0)\|)}{t},\quad \wh{\lambda}_\tau(S)=\sup_{A\in\mc{S}^\tau_\infty}\limsup_{t\to+\infty}\frac{\log(\|\Phi_A(t,0)\|)}{t}.
		\ee
	\end{defi}
The definition is independent of the choice of the norm $\|\cdot\|$ on $M_d(\R)$, since all norms on a finite-dimensional vector space are equivalent. 

Since \eqref{eq:lss} can be seen as a linear flow on a vector bundle, we have
\begin{equation}\label{two-exponents}
\lambda_\tau(S)=\wh{\lambda}_\tau(S)
\end{equation}
 for every choice of $\tau$ and $S$ \cite{Coloniusbook}.

	\subsubsection{Reduction to an irreducible component} We recall in this section how to reduce the stability analysis of \eqref{eq:lss} to the case in which the set ${S}$ is \emph{irreducible}, i.e., when there exist no proper subspaces of $\R^d$ that is invariant for all matrices in $S$.

The next result  
relates the properties of a reducible switched system to those of lower dimensional irreducible ones.  
For a discussion, see, for instance, \cite[Section~IV.B]{L2gain}.
	\begin{prop}\label{prop:reduction}
		Let ${S}\subset  M_d(\R)$ be bounded. 
		Then there exist $d_1\ge0$, $d_2>0$, and two subspaces $E_1\subsetneq E_2$ of $\R^d$ 
		of dimensions $d_1$ and $d_1+d_2$ 
		respectively,
		each of them invariant for all matrices in $S$, 
		such that for every basis $v_1,\dots,v_{d_1+d_2}$ of $E_2$ for which $v_1,\dots,v_{d_1}$ is a basis of $E_1$, representing  
for every matrix $A\in {S}$ the linear operator $A|_{E_2}$ in the basis $v_1,\dots,v_{d_1+d_2}$ as the matrix
		\be\label{eq:changeofbasis}
			A|_{E_2}=\left(   \begin{array}{cc}
									A_{11} & A_{12} \\
									0      & A_{22} 
																  \end{array}	
					 \right),
		\ee
		where each $A_{ij}$ is a $d_i\times d_j$ matrix, 
		we have that $\lambda_\tau(S)=\lambda_\tau(\{A_{22}\mid A\in S\})$ and $\{A_{22}\mid A\in S\}$ is irreducible. 
		Moreover, $E_1$ and $E_2$ can be chosen so that 
		either $d_1=0$ or $\lambda_\tau(\{A_{11}\mid A\in S\})<\lambda_\tau(S)$.
	\end{prop}

	\subsection{Projected switched systems and their orbits} It is an old idea  
	to describe the nonzero trajectories $x(t)$ of a linear system in coordinates $(\|x(t)\|,x(t)/\|x(t)\|)$.
	The advantage is that the dynamics of the {\it angular part} $x(t)/\|x(t)\|$ follow a well-defined switched system on a compact manifold. 
In order to formalize this idea, we associate with \eqref{eq:lss} the following projected switched system.
\begin{defi}\label{defi:othersys}
		Denoting by $\pi:\R^d\setminus\{0 \}\to\Prd$ the canonical projection of $\R^d$ onto its associated projective space and setting $s(t)=\pi(x(t))$, we define the nonlinear projected switched system 
		\be\label{eq:projsyst}\tag{$\pi\Sigma_\tau$}
			\dot{s}(t)=(\pi_*A(t))(s(t)),\qquad t\geq 0,\quad A\in\mc{S}^\tau_\infty, \quad s\in\Prd,
		\ee
		where for any matrix $A\in M_d(\R)$ we denote by $\pi_*A$ the projection of the vector field $ x\mapsto Ax $ onto $\Prd$.
	\end{defi}
	 
	Using the local identification $s=x/\|x\|$, 
	\eqref{eq:projsyst} can be rewritten as
	\be\label{eq:projsys}
		\dot{s}(t)=h(A(t),s(t))s(t),\quad \text{with}\quad h(A,s)=A-\langle s,A s\rangle \mathrm{Id}_d.
	\ee

			The system semigroup \eqref{eq:semigroup} associated with \eqref{eq:projsyst} is given by
	\be\label{eq:Smt}
		\Smt=\{ e^{t_mA_m}\circ\dots\circ e^{t_1A_1}\mid 
		m\in\N,\; A_1,\dots,A_m\in S,\;t_1,\dots,t_m\geq\tau\}.
	\ee
	We identify $\Smt$  with a semigroup both of  $\mathrm{GL}(\R,d)$ and of the group of diffeomorphisms of $\Prd$.
It is also useful to introduce the system group
	\be\label{eq:PP}
		\mc{P}_{S}=\{ e^{t_mA_m}\circ\dots\circ e^{t_1A_1}\mid m\in \N,\;A_1,\dots,A_m\in S,\;t_1,\dots,t_m\in\R\}.
	\ee
	It might be worth noticing that $\mc{P}_{S}$ would not be not affected if we added the dwell time constraint $|t_i|\geq \tau$, $i=1,\dots,m$. 
	Indeed, for every $A\in S$ and every $t\in\R$, we have $e^{t A}=e^{(2\tau+t)A}\circ e^{-2\tau A}$,
and if $|t|<\tau$, then $2\tau+t> \tau$.

	In addition to the previous notations, for every $s_0\in\Prd$ we denote the \emph{orbit 
		of $\mc{P}_{S}$ through $s_0$} by
	\be\label{eq:orbit}
		\mc{O}(s_0)=\{ \Phi (s_0)\mid \Phi \in \mc{P}_{S} \}.
	\ee

	The following result is a consequence of the analytic version of the Orbit Theorem \cite{Nagano,SussmannJurdjevic}. 
	
	\begin{prop}\label{prop:orbitanalytic}
		Let ${S}$ be a subset of $M_d(\R)$. Then the following properties hold true:
		\begin{enumerate}
			\item [i)]
			 For every $s_0\in\Prd$ the orbit $\mc{O}(s_0)$ is an immersed submanifold of $\Prd$. Moreover, for every $s\in \mc{O}(s_0)$, one has 
			\be\label{eq:tangspace}
				T_s\mc{O}(s_0)=\Lie_s( \pi_*{S} ),
			\ee
			where 
			$\pi_*{S}$ denotes the 
	set of vector fields on $\Prd$ defined by
	\[\pi_*{S}=\{ \pi_*A\mid A\in{S} \};\]
			\item [ii)] There exists $s_0\in\Prd$ for which the orbit $\mc{O}(s_0)$ is an embedded compact submanifold of $\Prd$;
			\item [iii)] If $S$ is irreducible then for every $s_0\in\Prd$ and every proper subspace $V\subset \R^d$ the orbit $\mc{O}(s_0)$ is not contained in $\mathbb{P}(V):=\{\pi x\mid x\in V,\,x\ne 0\}$. 
		\end{enumerate}
	\end{prop}
	
	\begin{proof}
		Point i) is a direct consequence of the analytic version of Orbit Theorem by Nagano and Sussmann, since $\Prd$ is an analytic manifold and each of the vector fields $\pi_*A$ is analytic on $\Prd$. 

As for point ii), the existence of $s_0$ such that $\mc{O}(s_0)$ is compact follows directly from Theorem~\ref{thm:closedorbit}, since 
$\mc{P}_{S}$
 is a Lie subgroup of $\GL(d,\R)$ (see, e.g., \cite[Propositions~2.6, 2.7]{Elliott}).
We then use the general fact that orbits which are closed as subsets of the ambient manifold are not only immersed but also embedded submanifolds (see \cite[Corollary 2.5]{BM97}). 
		
		We are left to prove iii). 
		Assume by contradiction that there exists $s_0=\pi x_0\in\Prd$ and a proper subspace $V\subset \R^d$ such that $\mc{O}(s_0)\subset \mathbb{P}(V)$.  
		Let 
		\[W={\rm span}\{ \Phi x_0\mid \Phi\in \mc{P}_{S} \}.\]
Then $W$ is contained in $V$ and is invariant for all matrices in $S$. Since $V$ is a proper subspace and $W\ne(0)$, this contradicts the irreducibility of $S$. 
	\end{proof}

	\begin{prop}\label{prop:notem}
		The interior of $\Smt$ for the relative topology on $\mc{P}_{S}$, seen as an immersed submanifold of ${\rm GL}(\R,d)$, is nonempty, that is,
		\[
			\mathrm{int}_{\mc{P}_{S}}(\Smt)\neq \emptyset.
		\]
	\end{prop}
	
	\begin{proof} 
		Consider 
		the system
		\begin{equation}\label{eq:lifted-sys}
		\dot{M}(t)=A(t) M(t)\qquad M(t)\in \mc{P}_{S},
		\end{equation}
		where $A(\cdot)\in \mc{S}^\tau_\infty$ is seen as a control law. 
		Such a system satisfies hypothesis \eqref{eq:mainhyp} on $\mc{P}_{S}$: indeed, 
		by the Orbit Theorem,  
		for every $Y\in \mc{P}_{S}$ the tangent $T_Y \mc{P}_{S}$ is equal to $\mathrm{Lie}_Y\{H\mapsto A H\mid A\in S\}$. (Notice that $H\mapsto A H$ is an analytic vector field on $\mathrm{GL}(\R,d)$ for every $A\in 
		M_d(\R)$.)
		
		We then apply Proposition~\ref{prop:notempty} to system \eqref{eq:lifted-sys},
		 to conclude that the attainable set $A_{\mathrm{Id}}^{T,\tau}$ from the identity map $\mathrm{Id}\in \mathrm{GL}(\R,d)$ has nonempty interior in $\mc{P}_{S}$ as soon as $T>d\tau$. This concludes the proof, since $\Smt$ contains $A_{\mathrm{Id}}^{T,\tau}$.
	\end{proof}

	\subsection{Uniqueness of the $\tau$-ICS in the projective space} 
	The main result of this section is Theorem~\ref{thm:uniq} stated below, which extends to the dwell-time setting the  uniqueness result for $0$-ICS of linear systems and the characterization of such a unique $0$-ICS (cf.~\cite{ArnoldKlie}). Let us mention that similar uniqueness results have been obtained (in the case $\tau=0$) for systems on Lie groups with a suitably defined linear structure (see \cite{Ayala2017}).

	\begin{thm}\label{thm:uniq}
	Assume that $S$ is an irreducible subset of $M_n(\R)$.  
		Let $s_0\in\Prd$ be such that the orbit $\mc{O}(s_0)$  
		is closed. Then there exists a unique $\tau$-ICS $D$ for 
		 system \eqref{eq:projsyst} contained in  $\mc{O}(s_0)$. Moreover, 
		\[
			D=\bigcap_{s\in\mc{O}(s_0)}\cl_{\mc{O}(s_0)}(\Smt(s))
		\]
		and $\mathrm{int}_{\mc{O}(s_0)}D\neq \emptyset$, that is, $D$ has nonempty interior in the orbit topology.
	\end{thm}
	
	\begin{remark}\label{rmk:Huniq}
	In the case where $\pi_* S$ satisfies assumption \eqref{eq:mainhyp}, then $S$ is irreducible and $\mc{O}(s_0)=\Prd$. In this case Theorem~\ref{thm:uniq} implies that  
	system \eqref{eq:projsyst} has a unique $\tau$-ICS in $\Prd$.
	\end{remark}
	
	Before providing a proof for Theorem~\ref{thm:uniq}, inspired by the one of \cite[Theorem 3.1]{ArnoldKlie}, let us present a couple of preliminary lemmas.
	
	\begin{lemma}\label{lemma:DTCS}
		Let $s_0\in\Prd$ be such that the orbit $\mc{O}(s_0)
		$ is closed, 
		and let us define
		\[
		D	=\bigcap_{ s\in\mc{O}(s_0) }\cl_{\mc{O}(s_0)}(\Smt(s)). 
		\]
		Assume that  $D\neq\emptyset$. Then $D$ is a $\tau$-ICS for 
		system \eqref{eq:projsyst} on 
		$\mc{O}(s_0)$.
	\end{lemma}
	
	\begin{proof}
		Since $D\subset \cl_{\mc{O}(s_0)}( \Smt(s) )$ for every $s\in D$ by construction, it is sufficient to show that $\Smt(s)\subset D$ for every $s\in D$. 
		
		Fix $s\in D$, $s'\in \Smt(s)$, and let us prove that $s'\in D$. By construction of $D$, $s$ is in $\cl_{\mc{O}(s_0)}(\Smt(r))$ for every $r\in\mc{O}(s_0)$. But this implies that also $s'$ is in $\cl_{\mc{O}(s_0)}( \Smt(r) )$ for every $r\in\mc{O}(s_0)$, that is, $s'\in D$.
	\end{proof}
	
	Given a linear subspace $V$ of $\R^d$, 
	we can identify any isomorphism of $V$ with a \emph{collineation of $V$}, i.e., a diffeomorphism of $\mathbb{P}(V)$, the projective space of $V$. Let us endow the group of collineations of $V$, denoted by $\mathrm{PGL}(V)$, with the topology of the uniform convergence.

We recall that a 
		 matrix $A\in \mathrm{GL}(\R,d)$ is \emph{semisimple} if and only if it admits a diagonal complex Jordan normal form.

	\begin{lemma}\label{lemma:om-lim-convex}
		Let $S$ be an irreducible subset of $M_d(\R)$ and $W\in \mc{P}_S$ be semisimple. Let $V\subset \R^d$ be the linear subspace spanned by the eigenspaces of $W$ corresponding to eigenvalues having maximal real part. 
		Fix $s_0\in \Prd$. Then
		\[\Gamma(s):=			\ov{\left\{ W^i(s)\mid i\in\N \right\}}\cap \mathbb{P}(V)
		\]
is nonempty for every $s$ in a dense subset of $\mc{O}(s_0)$. 
Assume moreover that $s_0$ is such that $\mc{O}(s_0)
		$ is closed. Then
$\mc{O}(s_0)\cap \mathbb{P}(V)$ is path-connected. 
\end{lemma}

\begin{proof}
Fix a system of coordinates in $\R^d$ such that $W$ is  in diagonal block form with blocks of size $1$ or $2$ and every block of size $2$ has the form $\begin{pmatrix}a &b\\-b&a\end{pmatrix}$ with $b\ne 0$. 
For the Euclidean structure associated with these coordinates, the orthogonal complement $V^\perp$ is spanned by the eigenspaces of $W$ corresponding to eigenvalues  whose real part is not maximal.

Consider a nonempty open subset $\Omega$ of $\mc{O}(s_0)$. 
		Notice that
		\[
			\Gamma(s)= \emptyset \quad \forall s\in \Omega
		\]
		only if $\Omega\subset \mathbb{P}(V^\perp)$. Since $\mc{O}(s_0)$ is an analytic manifold and $\Omega$ is nonempty and open, we deduce that $\mc{O}(s_0)\subset \mathbb{P}(V^\perp)$, which contradicts point  iii) of Proposition~\ref{prop:orbitanalytic}. This proves the first part of the statement.

Assume now that  $\mc{O}(s_0)$ is closed. In particular, $\mc{O}(s_0)\cap \mathbb{P}(V)$ is also closed.
Up to reordering  the coordinates of $\R^d$, we can assume that 
\[V=\{x\in\R^d\mid x_{r+1}=\dots=x_d=0\},\]
where $r$ denotes the dimension of $V$.
Endow $\mathbb{P}(V)$ with the homogeneous coordinates associated with $(x_1,\dots,x_r)$. Then, given $x\in \R^d\setminus V^\perp$, we can write $\pi((x_1,\dots,x_r,0,\dots,0))$ as 
$[x_1,\dots,x_{r}]$. 
Since $W$ is semisimple, Dirichlet's approximation theorem implies that there exists an unbounded sequence $\{ n_i \}_{i\in\N}\subset \N$ satisfying
		\begin{equation}\label{Dat}
		\lim_{i\to\infty}\left( W\big|_V \right)^{n_i}=\mathrm{Id}\big|_V\quad\mbox{in }\mathrm{PGL}(V).
		\end{equation}
		Hence, 
\begin{equation}\label{inOs}
[x_1,\dots,x_{r}]\in \Gamma(\pi(x))\subset \mc{O}(s_0)\cap \mathbb{P}(V)\qquad \forall x\in \R^d\setminus V^\perp\mbox{ such that }\pi(x)\in \mc{O}(s_0).
\end{equation}

Fix $s=\pi(x)$ and $\tilde s=\pi(\tilde x)$ in $\mc{O}(s_0)\cap \mathbb{P}(V)$. Consider an analytic path 
$\gamma:[0,1]\to \R^d$ such that $\gamma(0)=x$, $\gamma(1)=\tilde x$, and $\pi(\gamma(t))\in \mc{O}(s_0)$ for every $t\in[0,1]$. We claim that $t\mapsto [\gamma_1(t),\dots,\gamma_r(t)]$, which is defined for $t$ such that $\gamma(t)\not\in V^\perp$, admits a continuous extension $\hat\gamma:[0,1]\to \mathbb{P}(V)$. Indeed, if $t\in [0,1]$ is such that $\gamma(t)\in V^\perp$, let $m\in \N$ be the smallest positive integer such that the $m$-th derivative $\gamma^{(m)}(t)$ is not in $V^\perp$. Such $m$ exists by analyticity of $\gamma$. Hence $t$ is an isolated time for which $\gamma(t)\in V^\perp$ and
\[ [\gamma_1(\tau),\dots,\gamma_r(\tau)]= [(\gamma^{(m)}_1(t),\dots,\gamma^{(m)}_r(t))+O(t-\tau)]\quad\mbox{as $\tau\to t$.}\]
 
Hence, setting $\hat \gamma(t)=[(\gamma^{(m)}_1(t),\dots,\gamma^{(m)}_r(t))]$ we obtain a continuous extension $\hat\gamma:[0,1]\to \mathbb{P}(V)$ of $\gamma$. 
The values of $\hat \gamma$ are also in $\mc{O}(s_0)$, since the latter is closed and because of \eqref{inOs}. 
This concludes the proof of the lemma.
\end{proof}

	\begin{proof}[Proof of Theorem~\ref{thm:uniq}]
		We claim that there exists a semisimple matrix $W$ in $\mathrm{int}_{\mc{P}_{S}}(\Smt)$. In order to check it, recall that, since  $\mc{P}_{S}$ acts irreducibly on $\Prd$ (Proposition~\ref{prop:orbitanalytic}, point iii)), then all matrices in the Lie algebra $\mathfrak{p}_{S}$ of $\mc{P}_{S}$ out of a set of empty interior are semisimple (see \cite[Proof of Theorem 3.1., Step 1]{ArnoldKlie}). The conclusion then follows noticing that the exponential map $\exp:\mathfrak{p}_{S}\to \mc{P}_{S}$ is 
		a local diffeomorphism at every $y\in \mathfrak{p}_{S}$ and recalling that $\Smt$ has nonempty interior in $\mc{P}_{S}$.
		
		Observe that $\{ W^i\mid i\in\N \}\subset \mathrm{int}_{\mc{P}_{S}}(\Smt)$. Let $V\subset \R^d$ be the linear subspace spanned by the eigenspaces of $W$ corresponding to eigenvalues having maximal real part. 
		
		Fix a $\tau$-ICS $D\subset \mc{O}(s_0)$. 	Since $D$ has nonempty interior in $\mc{O}(s_0)$ we deduce from Lemma~\ref{lemma:om-lim-convex} that
 there exists $s_1 \in D$ such that 
 		\begin{equation*}
			\emptyset\neq \ov{\left\{ W^i(s_1)\mid i\in\N \right\}}\cap \mathbb{P}(V).
		\end{equation*}
In particular, 
\[D\cap \mathbb{P}(V)\neq \emptyset,\]
 since 
\[\ov{\left\{ W^i(s_1)\mid i\in\N \right\}}\subset \cl_{\mc{O}(s_1)}(\Smt(s_1))=D.\]

		Since $W$ is semisimple, Dirichlet's approximation theorem implies that
		there exists an unbounded sequence $\{ n_i \}_{i\in\N}\subset \N$ satisfying \eqref{Dat}.
Hence, for every $z\in\Z$, 
		\begin{equation*}
		\left(W\big|_V\right)^z=\lim_{i\to\infty} \left(W\big|_V\right)^{n_i+z}\in \cl_{\mathrm{PGL}(V)}\left(\left\{ \left(W\big|_V\right)^n\mid n\in\N \right\}\right)
		\end{equation*}
		since $n_i+z\ge 0$ for $i$ large enough.
Together with the closedness of $D$ and its dwell-time invariance, this implies that 
		\begin{equation*}
		W^z(D\cap \mathbb{P}(V))\subset D\cap \mathbb{P}(V)\quad \forall z\in\Z,
		\end{equation*}
and hence
		\begin{equation}\label{eq:ZbyN}
		W^z(D\cap \mathbb{P}(V))= D\cap \mathbb{P}(V)\quad \forall z\in\Z.
		\end{equation}

Let us now prove 
that 
	\begin{equation}\label{eq:PVD}
		\mc{O}(s_0) \cap \mathbb{P}(V) \subset D.
		\end{equation}
Since $W\in\mathrm{int}_{\mc{P}_{S}}(\Smt)$, it follows that 
		$W(s)\in \mathrm{int}_{\mc{O}(s_0)}D$ for every $s\in D$.
		In particular, according to \eqref{eq:ZbyN}, $D\cap \mathbb{P}(V)=W(D\cap \mathbb{P}(V))$ is contained in $\mathrm{int}_{\mc{O}(s_0)\cap \mathbb{P}(V)}D\cap \mathbb{P}(V)$. This means that $D\cap \mathbb{P}(V)$ is open (and closed) in the topology of $\mc{O}(s_0)\cap \mathbb{P}(V)$.
		Since, moreover,  $\mc{O}(s_0)\cap \mathbb{P}(V)$ is path-connected (Lemma~\ref{lemma:om-lim-convex}), we deduce that $D\cap \mathbb{P}(V)$ coincides with $\mc{O}(s_0)\cap \mathbb{P}(V)$. This completes the proof of 
		\eqref{eq:PVD}.

		The proof of the theorem can now be concluded. Indeed, 
		for every 
		$s\in\mc{O}(s_0)$ the set $\cl_{\mc{O}(s_0)}(\Smt(s))$ contains a $\tau$-ICS (Theorem~\ref{thm:existence}), which in turn 
		  contains  $\mc{O}(s_0)\cap\mathbb{P}(V)$ by \eqref{eq:PVD}. Hence,  $\bigcap_{s\in
			\mc{O}(s_0)}\cl_{\mc{O}(s_0)}(\Smt(s))$ is nonempty, which implies that it is a $\tau$-ICS contained in $\mc{O}(s_0)$, thanks to Lemma~\ref{lemma:DTCS}. 
Moreover, it is the unique one, since  two $\tau$-ICS with nontrivial intersection coincide (as it follows immediately from Definition~\ref{defi:DTCS}).

The last part of the statement follows from Theorem~\ref{thm:existence}.
	\end{proof}
	
	Let us then consider the $\tau$-ICS
	\[
		D=\bigcap_{s\in\mc{O}(s_0)}\cl_{\mc{O}(s_0)}(\Smt(s))\subset \mc{O}(s_0).
	\]
	By point iv) of Lemma~\ref{lemma:properties}, there exists an open and dense set (with respect to the induced orbit topology) $\mathfrak{C}\subset D$, such that
	\be\label{eq:C}
		\mathfrak{C}=\Smt(s),\quad \text{for every }s\in\mathfrak{C}.
	\ee
	As a corollary of Theorem~\ref{thm:uniq} we may now deduce the following useful result.

	\begin{cor}\label{cor:uniftime}
		Let $s_0\in\Prd$ be such that the orbit $\mc{O}(s_0)$
		is closed and let $D$ 
		be the unique $\tau$-ICS contained in $\mc{O}(s_0)$. Let $\mathfrak{C}\subset D$ be the open and dense set satisfying \eqref{eq:C}. Then, for every $s_1\in\mathfrak{C}$, there exists $T>0$ such that
		\[
			s_1\in \Ats{T}{s_2}
		\]
		for every $s_2\in\mc{O}(s_0)$.
	\end{cor}

	\begin{proof}
		Fix $s_1\in\mathfrak{C}$ and  
		pick 
		any $s_2\in\mc{O}(s_0)$. 
		Corollary~\ref{cor:revers} implies that there exist a neighborhood $W_{s_2}$ of $s_2$ in $\mc{O}(s_0)$, a time $T'>0$, and a point $s_3\in  \mc{O}(s_0)$ such that for every $s_2'\in W_{s_2}$, 
		\[s_3\in \Ats{T'}{s_2'}.\]

		Since $
		D\subset \cl_{\mc{O}(s_0)}(\Smt(s_3))$ by Theorem~\ref{thm:uniq} and $\mathfrak{C}$ is open and dense in $D$, then there exist a signal $A'(\cdot)\in\mc{S}
^\tau$ and $s_4\in 
		\mathfrak{C}
		$ such that
		\[
			s_4=\Phi_{A'}(T(A'),0)(s_3).
		\]
		Moreover, since $\mathfrak{C}=\Smt(s)$ for every $s\in\mathfrak{C}$, then there exists another signal $A''(\cdot)\in\mc{S}^\tau$ such that
		\[
			s_1=\Phi_{A''}(T(A''),0)(s_4).
		\]
		We combine together these identities, concatenating the corresponding signals and concluding that 
		\[s_1\in \Ats{T'+T(A')+T(A'')}{s_2'}\qquad\mbox{for all }s_2'\in W_{s_2}.\]
		Finally, by extracting a finite covering of  $\mc{O}(s_0)$ by neighborhoods of the type $W_{s_2}$, we conclude the proof of the uniformity of $T$ as in the statement.
	\end{proof}

	\subsection{Periodization} \label{sec:periodization}
	
We are now ready to present a result on the uniform exponential rate $\lambda_\tau(S)$ of $\Sigma_{\tau}$ introduced in Definition~\ref{defi:unifexprate}. Let us define
	\[
		\lambda^{\rm per}_\tau(S)=\sup_{(A,x_0)\in\mc{S}_{\rm per}}\limsup_{t\to+\infty}\frac{\log(\|\Phi_A(t,0)x_0\|)}{t},
	\] 
	where $\mc{S}_{\rm per}$ consists of the pairs $(A(\cdot),x_0)\in\mc{S}^\tau_\infty\times (\R^d\setminus\{0\})$ for which there exists $T>0$ such that both $A(\cdot)$ and $t\mapsto \pi\Phi_A(t,0)x_0$ are $T$-periodic. 
	
	The main result of this section is the following.

	\begin{thm}\label{thm:unifper}
		Let $S\subset M_d(\R)$ be a bounded 
		set of matrices and let $\tau\ge 0$. Then
		\[
			\lambda^{\rm per}_\tau(S)=\lambda_\tau(S).
		\]
	\end{thm}
	
	\begin{proof}
		It is sufficient to prove that $\lambda_\tau(S)\leq \lambda^{\rm per}_\tau(S)$, the other inequality being obvious by definition. 
	
	Let us first show that it is enough to prove the theorem when $S$ is irreducible. 
	By Proposition~\ref{prop:reduction}, there exists an invariant subspace $E$ of $\R^d$ 
	such that, up to a linear change of coordinates, 
for every $A\in {S}$, 
		\[ 
			A|_{E}=\left(   \begin{array}{cc}
									A_{11} & A_{12} \\
									0      & A_{22} 
																  \end{array}	
					 \right),
		\]
	with  $S_{22}=\{A_{22}\mid A\in S\}$  irreducible and 
	$\lambda_\tau(S)=\lambda_\tau(S_{22})$.  
	We should prove that $\lambda_\tau(S)\le \lambda^{\rm per}_\tau(S)$ knowing that 
	$\lambda^{\rm per}_\tau(S_{22})=\lambda_\tau(S_{22})$.
		Notice that 
		when the block $A_{11}$ has dimension $0\times 0$, then $\lambda^{\rm per}_\tau(S_{22})\le \lambda^{\rm per}_\tau(S)$ and the conclusion follows.
According to the last part of the statement of Proposition~\ref{prop:reduction}, we can then assume that
\[\lambda^{\rm per}_\tau(S_{11})\le \lambda_\tau(S_{11}) <\lambda_\tau(S_{22})=\lambda^{\rm per}_\tau(S_{22}),\]
where $S_{11}=\{A_{11}\mid A\in S\}$. 		 
	Set $S|_E=\{A|_{E}\mid A\in S\}$. By invariance of $E$ we have that 
	 $\lambda^{\rm per}_\tau(S|_E)\le \lambda^{\rm per}_\tau(S)$. We are concluding the argument for the reduction to the irreducible case by showing that  $\lambda^{\rm per}_\tau(S_{22})\le \lambda^{\rm per}_\tau(S|_E)$.

The tricky point is to show that a periodic-in-projection trajectory of $(\Sigma^\tau_{S_{22}})$ can be lifted to a periodic-in-projection trajectory of $(\Sigma^\tau_{S|_E})$. 
Let $A\in \mc{S}^\tau$, $\mu\in \R$, and $x_2$ be such that  
\[ \Phi_{A_{22}}(T(A),0)x_2=\mu x_2,\qquad \log|\mu|>\lambda^{\rm per}_\tau(S_{11}).\]
Then, for every $x_1$ such that $x=(x_1,x_2)$ is in $E$, 
\[ \Phi_{A|_{E}}(T(A),0)x=
\begin{pmatrix}
\Phi_{A_{11}}(T(A),0)x_1+\int_0^{T(A)}
\Phi_{A_{11}}(T(A),s)A_{12}(s)
\Phi_{A_{22}}(s,0)x_2 ds\\ \mu x_2\end{pmatrix}\]
and $\mu$ is not in the spectrum of $\Phi_{A_{11}}(T(A),0)$. 
The lift is then obtained by choosing as $x_1$ the solution to
 \[
 \mu x_1= 
 \Phi_{A_{11}}(T(A),0)x_1+\int_0^{T(A)} 
 \Phi_{A_{11}}(T(A),s)A_{12}(s)
 \Phi_{A_{22}}(s,0)x_2 ds.
 \]
This concludes the proof of the reduction to the case where $S$ is irreducible.
		
		Equality \eqref{two-exponents} implies that for every $\varepsilon>0$ there exist $A(\cdot)\in \mc{S}^\tau_\infty$ and $\ov{t}=\ov{t}(\varepsilon)>0$ such that 
		\begin{equation}\label{-eps}
			\frac{\log(\| \Phi_A(t,0)  \|)}{t}>\lambda_\tau(S)-\varepsilon, \qquad\mbox{for every $t\ge \ov{t}$}.
		\end{equation}

 Let $s_0\in\Prd$ be such that $\mc{O}(s_0)$ is closed, and $D$ be the unique $\tau$-ICS contained in $\mc{O}(s_0)$ given by Theorem~\ref{thm:uniq}.  
		As recalled above, there exists an open set $\mathfrak{C}$ for the topology of $\mc{O}(s_0)$ such that $\cl_{\mc{O}(s_0)}(\mathfrak{C})=D$. 
		We claim that we can find $d$ linearly independent unit vectors $v_1,\dots,v_d\in\R^d$ such that their projections $s_i=\pi(v_i)$ belong to $\mathfrak{C}$ for every $i=1,\dots,d$.
		Indeed, since $\mc{O}(s_0)$ is an analytic orbit and because of the openness of $\mathfrak{C}$, then
		\[ {\rm span}\{x\in \R^d\setminus\{0\}\mid \pi x\in \mathfrak{C}\}={\rm span}\{x\in \R^d\setminus\{0\}\mid \pi x\in \mc{O}(s_0)\}=\R^d,\]
		where the last equality follows from Proposition~\ref{prop:orbitanalytic}, point iii).

		Let then $v_1,\dots,v_d\in\R^d$ be chosen as above, and notice that the map $B\mapsto\max_{i=1,\dots,d}\| Bv_i \|$ is a norm on $M_d(\R)$. 
		Hence, there exists a constant $C>0$ such that $\max_{i=1,\dots,d}\| Bv_i \|\ge C\|B\|$ for every $B\in M_d(\R)$. 
		Using this property and \eqref{-eps}, we deduce that for every $\varepsilon>0$ there exist 
		$A(\cdot)\in \mc{S}^\tau$ and $l\in\{ 1,\dots, d \}$ 
		such that $T(A)>1/\varepsilon$ and 
		\[
			\frac{\log(\| \Phi_A(T(A),0)v_l  \|)}{T(A)}>\lambda_\tau(S)-\varepsilon.
					\]
		By Corollary~\ref{cor:uniftime} there exists a time $T>0$ independent of $\varepsilon$ and a signal 
		$A'\in \mc{S}^\tau$ (possibly depending on $\varepsilon$) 	such that  $T(A')<T$ and 
		\[
			v_l=\Phi_{A'}(T(A'),0)
			\Phi_A(T(A),0)v_l.
		\]
		Let then $A''(\cdot)$ be the 
		$(T(A)+T(A'))$-periodic signal contained in $\mc{S}_{\infty}^\tau$ defined by the concatenation of $A$ and $A'$ on each period. By construction, the pair $(A'',v_l)$ is in $\mc{S}^\tau_{\rm per}$.  Moreover, if  $\varepsilon\ll 1$ then $T(A)\gg T>T(A')$, which implies that 
		\[
			\frac{\log( \|\Phi_{A''}(T(A)+T(A'),0)v_l\| )}{T(A)+T(A')}>\lambda_\tau(S)-2\varepsilon.
		\]
This readily leads to the conclusion by the arbitrariness of $\varepsilon$.
	\end{proof}

	\section{Piecewise deterministic dwell-time random processes}
	\label{s:stoch}

	In this section we apply the theory developed so far to a class of piecewise deterministic dwell-time random processes. Inspired by \cite{BenaimColonius,BenaimLeborgne}, we show how $\tau$-ICS are naturally related to the support of the invariant measures associated with such processes.
	
	\subsection{General constructions}
	
	Let us consider a compact manifold $M$, a finite set of indices $E=\{ 1,\dots, m \}$, $m\ge 2$, and a family $\mc{F}=\{ X_i\mid i\in E \}$ of smooth vector fields on $M$. Moreover, let $\mc{Q}:M\to M_m(\R)$, $\mc{Q}:q\mapsto (\mc{Q}(q,i,j))_{i,j\in E}$ be a continuous map such that $\mc{Q}(q)$ is a  Markov transition matrix  and $\mc{Q}(q,i,j)>0=\mc{Q}(q,i,i)$ for every $q\in M$ and every $i,j\in E$, $i\ne j$.
	
	Given $\lambda>0$ and $\tau> 0$, let $(U_i)_{i\ge 1}$ be a sequence of i.i.d. random variables with real positive values, whose density $f=f_{U_i}:(0,\infty)\to [0,\infty)$ is exponential of intensity $\lambda$ up to a right-shift by $\tau$, that is,
	\be\label{eq:density}
		f(t)=\lambda e^{-\lambda(t-\tau)}\mathbf{1}_{\{t\ge \tau\}}.
	\ee
	Let $0=T_0<T_1<\dots<T_n<\dots$ be the sequence of random points in $[0,+\infty)$
	defined  by $T_i=U_1+\dots+U_i$, $i\ge 1$.

	Finally,  let $(N_t)_{t\ge 0}$ be the counting process associated with $(T_n)_{n\ge 0}$, for which we have the standard relation
	\[
		\PP(N_t\ge n)=\PP(T_n\le t),\quad t\in\R,\quad n\in\N.
	\]
Notice that, almost surely, 
\be\label{Nt/t}\frac{T_n}n\to \tau+\frac1\lambda=\frac{\tau \lambda+1}\lambda,\qquad \frac{N_t}t\to \frac\lambda{\tau \lambda+1}.
\ee 
	
	Consider now a random variable $Z_0$ on $M\times E$, independent of the process $(U_i)_{i\in\N}$, and construct $Z_n=(Q_n,L_n)$ on $M\times E$ inductively by
	\begin{align}
		Q_{n+1}&=e^{U_{n+1}X_{L_n}}(Q_n),\\
		\PP( L_{n+1}&=j|Q_{n+1},L_n=i )=\mc{Q}(Q_{n+1}, i,j ).
	\end{align}
The process $Z_n$ has, by construction, the Markov property.

\begin{remark}
In analogy with \cite{BenaimLeborgne}, one could further generalize the above construction by allowing $\lambda$ to be a uniformly positive function depending on the indices $L_n$ and $L_{n+1}$ and on the point $Q_n$. 
Another possible type of generalization, following \cite{ColoniusMazanti}, would consist in allowing more general probability densities than $f$, supported in $[\tau,\infty)$.
We prefer to restrict the framework in order to keep notations reasonably simple. 
The constraint $\mc{Q}(q,i,i)=0$ reflects the assumption that the distribution of the duration of the bangs follows a shifted exponential law.
\end{remark}

We find it useful to introduce the continuous-time process $(Y_t)_{t\ge 0}$ obtained by interpolation of $Q_n$, as follows, 
	\be\label{eq:proc}
		Y_t=e^{(t-T_n)X_{L_n}}(Q_n)\quad \text{for every } t\in [T_n,T_{n+1}).
	\ee

	\subsection{Invariant measures}

	By a classical result of Krylov and Bogolyubov \cite{KriBog}  (see also, for instance, \cite{Duflo1997}), 
compactness of $M$ and theFeller property
		for the process $(Z_n)_{n\in\N}$ 
	imply that	 there exists at least one invariant measure for the process $(Z_n)_{n\in\N}$ described in the previous section.

		For every $(q,i)\in M\times E$ and every measurable set $A\subset M\times E$, we define the $n$-step transition probability from $(q,i)$ to $A$ as
		\begin{equation}\label{defi:transprob}
			P_n((q,i),A)=\E [ Z_n\in A| Z_0=(q,i) ].
		\end{equation}
		
	For every invariant measure $\mu$ and every measurable set $A\subset M\times E$, we then have
	\be\label{eq:meas}
		\mu(A)=\int_{M\times E}P_n((q,i),A)d\mu(q,i)=\int_{\mathrm{supp}\mu}P_n((q,i),A)d\mu(q,i),\quad \text{for every }n\ge 1.
	\ee
	
	Following \cite{BenaimColonius}, we define, for each $n\in\N$, the sets
	\begin{align*}
		\mathbf{T}_n&=\{ (\ib,\ub)=((i_0,\dots,i_n),(u_1,\dots,u_n))\in E^{n+1}\times [\tau,+\infty)^n \mid i_{k-1}\ne i_{k}\mbox{ for }k=1,\dots,n\},\\
		 \mathbf{T}_n^{ij}&=\{ (\ib,\ub)\in\mathbf{T}_n\mid i_0=i,i_n=j \}.
	\end{align*}
	The trajectory $\phi(q,t,\ib,\ub)$, induced by a pair $(\ib,\ub)\in\mathbf{T}_n$ and starting at $q\in M$ is then determined as follows: let $t_0=0$, $t_k=t_{k-1}+u_k$ for $1\leq k\leq n$ and set $q_0=q$, $q_{k}=e^{u_kX_{i_{k-1}}}(q_{k-1})$ for $1\leq k\leq n$. Then
	\be\label{eq:phi}
		\phi(q,t,\ib,\ub)=\begin{cases}
			q, \quad t=0,\\
			e^{(t-t_{k-1})X_{i_{k-1}}}(q_{k-1}),\quad t_{k-1}<t\leq t_k,\\
			e^{(t-t_n)X_{i_n}}(q_n),\quad t>t_n.
		\end{cases}
	\ee
	
 Notice that $\phi(q,t,\ib,\ub)$ is in fact a trajectory of \eqref{eq:controlsys} driven by a piecewise constant control  
 with dwell time $\tau$, taking $U=E$. 
	
	\begin{lemma}\label{lemma:tubelemma}
		Let $(\ib,\ub)\in\mathbf{T}_n$. Then, for every $q\in M$, $T\ge 0$, and $\delta>0$,
		there exist $\beta>0$ and a neighborhood $W_q$ of $q$ such that 
		\[
			\PP\left( \sup_{0\le t\le T}\| Y_t-\phi(z,t,\ib,\ub) \|\le\delta\mid Z_0=(z,i_0) \right)>\beta\qquad\forall z\in W_q,
		\]
		where $(Y_t)_{t\ge0}$ is defined as in \eqref{eq:proc}.
	\end{lemma}

	\begin{proof}
		The proof goes along the same lines as \cite[Lemma 3.2]{BenaimLeborgne}, therefore we only point out the necessary modifications. 

		The proof that the deterministic trajectory $\phi(z,t,\ib,\ub)$ can be approximated by stochastic trajectories $Y_t$ is much simpler here than in  \cite{BenaimLeborgne}, since we are assuming that all off-diagonal entries of the Markov transitioning matrix $\mc{Q}$ are strictly positive.
		Moreover, a simple computation shows that, if $W$ is a random variable whose density is as in \eqref{eq:density}, and $w\in [\tau,+\infty)$, then
		\be\label{eq:estf}
			\PP(|W-w|\leq \delta)=e^{-\lambda( \max\{ w-\delta-\tau,0 \} )}-e^{\lambda(w+\delta-\tau)}>0.
		\ee
This allows to conclude as in the aforementioned reference.
	\end{proof}

		\begin{lemma}\label{lemma:secondincl}
		Let $q\in M$, $i\in E$, $y\in\ov{\Sft(e^{\tau X_i}q)}$, and let $W_y\subset M$ be a neighborhood of $y$. Then there exist 
		$n\in \N$, $\beta>0$, $j\in E$, and
		a neighborhood $W_q\subset M$ of $q$ such that  
		\[
			P_{n}((z,i), W_y\times\{ j \} )>\beta\qquad \forall z\in W_q.
			\]
	\end{lemma}
	
	\begin{proof}
The lemma is a direct consequence of Lemma~\ref{lemma:tubelemma}, once we notice that, by assumption, there exist $y_0\in \mathrm{int}(W_y)$ and $g^i\in\Sft$ such that $y_0=g^i(e^{\tau X_i}q)$. 
			\end{proof}

	\begin{lemma}\label{lemma:opencover}

Assume 
that $M$ is compact and that $\mc{F}$ satisfies assumption \eqref{eq:mainhyp}.	
 Let 
		 $D_1,\dots,D_l$ be the $\tau$-ICS associated with \eqref{eq:controlsys}. 
		Then, for every $q\in M$, there exist 
		$n\in \N$, $\beta>0$, and 
		a neighborhood $W_q\subset M$ of $q$ such that, for every $i\in E$ and every $z\in W_{q}$,
		\[
			P_{n}\left((z,i),\bigcup_{r=1}^l\mathrm{int}(D_r)\times E\right)>\beta.
		\]
	\end{lemma}
	
	\begin{proof}
		Let $q\in M$ and $i\in E$. By Theorem~\ref{thm:existence}, there exists $1\le r\le l$ such that $D_r\subset \ov{\Sft(e^{\tau X_i}(q))}$ and $D_r$ has nonempty interior. 
		Lemma~\ref{lemma:secondincl} then implies that there exist $n^{i}\in\N$, $\beta^i>0$, and 
		a neighborhood $W_q^i\subset M$ of $q$ such that
		\[
			P_{n^i}\left((z,i),\bigcup_{r=1}^l\mathrm{int}(D_r)\times E\right)>\beta^i\qquad \forall z\in W^i_{q}.
		\]
		Observe that $n^{i}$ can be made uniform with respect to $i\in E$. Indeed, due to point (ii) of Lemma~\ref{lemma:properties}, each set $\mathrm{int}(D_r)$ is dwell-time positively invariant, and therefore 
		\[
				P_{n}\left((z,i),\bigcup_{r=1}^l\mathrm{int}(D_r)\times E\right)\ge 
				P_{n^i}\left((z,i),\bigcup_{r=1}^l\mathrm{int}(D_r)\times E\right)\qquad \forall n\ge n^i.\]
		The proof is then concluded taking
		$n=\max_{i\in E}n^i$, $\beta=\min_{i\in E}\beta^i$, and $W_q=\cap_{i\in E}W_q^i$.
	\end{proof}

The following proposition is an adaptation of \cite[Theorem 4.5]{BenaimColonius}.	
	\begin{prop}\label{prop:firstest}
	Assume 
that $M$ is compact and that $\mc{F}$ satisfies assumption \eqref{eq:mainhyp}.	
 Let $D_1,\dots,D_r$ be the $\tau$-ICS associated with \eqref{eq:controlsys}.
		Then, for every invariant measure $\mu$ of $(Z_n)_{n\in\N}$, one has that 
		\[
			\mathrm{supp}\mu\subset \bigcup_{r=1}^lD_r\times E.
		\]
	\end{prop}

	\begin{proof}
		Suppose, by contradiction, that there exists 
		\[
			(q_0,i_0)\in A:=\mathrm{supp}\mu\setminus\left( \bigcup_{r=1}^l D_r\times E \right).
		\]
		Since each $D_r$ is closed, there exists an open neighborhood $W_{q_0}\subset M$ of $q_0$ with 
		\[W_{q_0}\bigcap \left( \bigcup_{r=1}^l D_r \right)=\emptyset.\]
		In particular, 
		$A\cap(W_{q_0}\times E)=\mathrm{supp}\mu \cap(W_{q_0}\times E)$,
		which implies that
		\be\label{measure>0}
		\mu(A\cap(W_{q_0}\times E))=\mu(W_{q_0}\times E)>0,
		\ee
		where the inequality follows from the fact that $(q_0,i_0)$ is in $\mathrm{supp}\mu$.
		Owing to the invariance of $\mu$ and
		the dwell-time positively invariance of $
		D_r
		$, 
we have
		\begin{align}
			\mu(A)&=\int_{\mathrm{supp}\mu}P_{n
				+1}((q,i),A)d\mu(q,i)\\
			&=\underbrace{\int_{\mathrm{supp}\mu\cap(\bigcup D_r\times E)}P_{n
					+1}((q,i),A)d\mu(q,i)}_{=0}+\int_{\mathrm{supp}\mu\setminus(\bigcup D_r\times E)}P_{n
				+1}((q,i),A)d\mu(q,i)\\
			&=\int_AP_{n
				+1}((q,i),A)d\mu(q,i)\\&\le \int_A\left(1-P_{n
				} \left( (q,i),\bigcup_{r=1}^l\mathrm{int}(D_r)\times E \right) \right)d\mu(q,i)\\
			&\le \mu(A)-\int_{A\cap(W_{q_0}\times E)}P_{n
			} \left( (q,i),\bigcup_{r=1}^l\mathrm{int}(D_r)\times E \right)d\mu(q,i).
		\end{align}
	Using \eqref{measure>0} and Lemma~\ref{lemma:opencover},  
	we get that
		\[\int_{W_{q_0}\times E} P_{n
				}\left((q,i),\bigcup_{r=1}^l \mathrm{int}(D_r)\times E\right)d\mu(q,i)>0,
		\]
which leads to a contradiction. 
	\end{proof}
		
Denote by $\pi_M:M\times E\to M$ the projection $\pi_M:(q,i)\mapsto q$.

	\begin{prop}\label{prop:revsupp}
		Let $M$ be a compact manifold and assume that $\mc{F}$ satisfies assumption \eqref{eq:mainhyp}. Let $D$ be a $\tau$-ICS associated with \eqref{eq:controlsys}, and let $\mu$ be any invariant measure associated with the process $(Z_n)_{n\ge 0}$.
Assume that $\mathrm{supp}\mu\cap (
		D\times E)\neq \emptyset$. Then $D\subset \pi_M(\mathrm{supp}\mu)$ and, in particular, $\mu(D\times E)>0$.
	\end{prop}
	
	\begin{proof}
		Assume, by contradiction, that there exists $y\in D$ such that $(y,j)\notin \mathrm{supp}\mu$ for every $j\in E$. Since $D=\ov{\mathrm{int}(D)}$ and $\pi_M(\mathrm{supp}\mu)$ is closed, it is not restrictive to assume $y\in\mathrm{int}(D)$, whence there exists an open neighborhood $W_y\subset \mathrm{int}(D)$ satisfying
		\be\label{eq:supports}
			(W_y\times \{j\})\cap \mathrm{supp}\mu=\emptyset\qquad\mbox{for every }j\in E.
		\ee
		Let now $(q_0,i_0)\in \mathrm{supp}\mu\cap(D\times E)$. 
		As a consequence of the equality $\ov{\Sft(p)}=D$ for every $p\in D$, it follows both that $e^{\tau X_i}q_0\in 
		D
		$ for every $i\in E$, 
		and
		\[
			y\in\bigcap_{i\in E}\ov{\Sft(e^{\tau X_i}q_0)}.
		\]
By	Lemma~\ref{lemma:secondincl}, there exist an open neighborhood $W_{q_0}\subset M$ of $q_0$ and a 
map $i\mapsto (n_i,j_i)\in (\N\setminus\{0\})\times E$ such that 
	\[
			P_{n_i}((z,i),W_y\times \{ j_i \})>0\qquad\mbox{for  every $z\in W_{q_0}$}.
	\]
Notice that, since $(q_0,i_0)\in\mathrm{supp}\mu$, then 
 		\[
			0< \mu(W_{q_0}\times \{i_0\}). 
		\] 
Hence,
\begin{align*}
		\sum_{l\in E}	\int_{M\times E}P_{n_{l}}((z,i),W_y\times\{j_i\})d\mu(z,i)&\ge\int_{M\times E}P_{n_{i_0}}((z,i),W_y\times\{j_i\})d\mu(z,i)\\
&		\ge \int_{M}P_{n_{i_0}}((z,i_0),W_y\times\{j_{i_0}\})d\mu(z,i_0)>0.
		\end{align*}

		Now, from \eqref{eq:meas} we have
		\begin{align}\label{eq:eq1}
			1&=\mu(M\times E)=\frac{1}{m}\sum_{l\in E}\int_{M\times E} P_{n_l}((z,i),M\times E)d\mu(z,i)\\
			&=\frac{1}{m}\sum_{l\in E}\left( \int_{M\times E} P_{n_l}((z,i),(M\times E)\setminus (W_y\times\{j_i\}))d\mu(z,i)+\int_{M\times E} P_{n_l}((z,i),W_y\times \{j_i\})d\mu(z,i) \right)\\
			&>\frac{1}{m}\sum_{l\in E}\int_{M\times E} P_{n_l}((z,i),(M\times E)\setminus (W_y\times\{j_i\})d\mu(z,i).
		\end{align} 
This leads to a contradiction, since
		\begin{align}
			1&=\mu(\mathrm{supp}\mu)=\frac{1}{m}\sum_{l\in E}\int_{M\times E}P_{n_l}((z,i),\mathrm{supp}\mu)d\mu(z,i)\\
			&\le \frac{1}{m}\sum_{l\in E}\int_{M\times E}P_{n_l}((z,i), (M\times E )\setminus (W_y\times \{j_i\}) d\mu(z,i)<1,
		\end{align}
		where the first inequality follows from   \eqref{eq:supports}.
	\end{proof}
	
	\subsection{Ergodic invariant measures
	}

	We consider now the case of invariant ergodic measures.
	
	\begin{defi}
		An invariant measure $\mu$ for the discrete-time process $(Z_n)_{n\in\N}$ is said to be \emph{ergodic} if it cannot be expressed as a proper convex combination of invariant measures for the same process.
	\end{defi}
	
In analogy with \cite[Theorem 4.7]{BenaimColonius}, we have the following result.
	
	\begin{thm}\label{thm:ergodic}
		Let $M$ be compact manifold and assume that $\mc{F}$ satisfies \eqref{eq:mainhyp}. Then
		\begin{itemize}
			\item [i)] For every ergodic invariant measure $\mu$ of the discrete-time process $(Z_n)_{n\in\N}$ there is a $\tau$-ICS $D$ for which $\pi_M(\mathrm{supp}\mu)=D$.
			\item [ii)] Conversely, let $D$ be any $\tau$-ICS. Then there exists an ergodic invariant measure $\mu$ with $\pi_M(\mathrm{supp}\mu)=D$.
			Moreover, $\mu$ is absolutely continuous with respect to the Lebesgue measure, and is the unique ergodic invariant measure whose support is contained in $D\times E$.
		\end{itemize}
	\end{thm}
	
	\begin{proof}
		Let us first prove i). Proposition~\ref{prop:firstest} implies that $\pi_M(\mathrm{supp}\mu)$ is contained in the union of all $\tau$-ICS $D_1,\dots,D_l$. 
		By Proposition~\ref{prop:revsupp} we only need to show that that $\mathrm{supp}\mu$ intersects only one set of the form $
		D_r\times E$, $r\in\{1,\dots,l\}$. 
		Let $r$ be such that $\mathrm{supp}\mu\cap (D_r\times E)\ne 0$. Then $\mu(D_r\times E)>0$ by Proposition~\ref{prop:revsupp} and, for every $\mu$-measurable set $A\subset 
		D_r\times E$, we have 
		\[
			\mu(A)=\sum_{s=1}^l\int_{D_s\times E}P_n((q,i),A)d\mu(q,i),\quad \text{for every }n\in\N\setminus\{0\}.
		\]
		Notice that it actually sufficient to integrate over $D_r\times E$, since $(q,i)\in\mathrm{supp}\mu\cap(D_s\times E)$ implies that  $q\in D_s$. If $s\neq r$, by the dwell-time positive invariance of $D_s$, it is 
		impossible to connect $(q,i)$ to $A$ by an admissible trajectory of \eqref{eq:controlsys}.
		
		It follows that the restriction of $\mu$ to $D_r\times E$ is an invariant measure, and therefore one must also have 
		\[
			\mu\left(\bigcup_{s\neq r}D_s\times E\right)=0,
		\]
		for otherwise the same reasoning as before would imply that $\mu$ could be written as a proper convex sum of invariant probability measures, contradicting the ergodicity assumption.
		
		As for point ii), we observe that the existence of an invariant  measure whose support is contained in $D\times E$ 
		 relies on the Feller property
		for the process $(Z_n)_{n\in\N}$ and the compactness and the positive dwell-time invariance of $
		D\times E$. 
		Since the set of invariant measures with support contained in $D\times E$ is convex and compact, by the Krein--Milman theorem 
		it contains at least one ergodic invariant measure $\mu$, and the fact that its support has projection equal to $D$ then follows from point i). 
		The dwell-time invariance of $D\times E$ 
		and hypothesis  \eqref{eq:mainhyp} allow to conclude as in 
		\cite[Theorem 1]{Bakhtin} (see also \cite[Section 4]{BenaimLeborgne}) that $\mu$ is the unique ergodic invariant measure with support contained in $D\times E$ and that $\mu$ is absolutely 
		continuous with respect to the Lebesgue measure.
		\end{proof}

\subsection{Stochastic Lyapunov exponents of dwell-time linear systems}\label{s:stoch-lin}

Assume in this section that $M=\Prd$ and that the vector fields $X_i$, $i\in E=\{1,\dots,m\}$, are induced by linear matrices $A_i$, $i\in E$.
For simplicity, we also assume that the transition matrix $\mc{Q}$ is independent of $q$. Let the processes $T_n$, $U_n$, and $L_n$ be defined as in the previous section and associate with them the process $\Phi_n$ 
defined 
recursively by $\Phi_0=\mathrm{Id}\in \mathrm{GL}(\R,d)$ and $\Phi_{n+1}=e^{U_{n+1} A_{L_n}}\Phi_n$.

Under the assumption that $\{\pi_*A_i\mid i\in E\}$ satisfies \eqref{eq:mainhyp}, 
we denote by $D\subset \Prd$ the  unique $\tau$-ICS for system \eqref{eq:projsyst} (see Remark~\ref{rmk:Huniq}) and by 
	$\mu$ the unique invariant ergodic measure  provided by Theorem~\ref{thm:ergodic}.
	Denote by $\mu_0$ the probability distribution of $Z_0$ on $\Prd\times E$. 
The sequence of probability measures
 \[A\mapsto \int_{\Prd\times E} P_n((q,i),A)d\mu_0(q,i)\]
then converges to $\mu$ in total variation distance as $n\to \infty$ (see \cite[Theorem 4.5]{BenaimLeborgne}). 	

Thanks to Furstenberg--Kesten theorem (see, for instance, \cite{Viana} or \cite{Arnold1998Random}),  
there exists $\chi_d$ in $\R$ such that, almost surely, 
\be\label{eq:chid}
	\lim_{n\to\infty}\frac{1}{n}\log\left\|\Phi_n\right\|=\chi_d.
\ee

Moreover, by \cite[Proposition 3.12]{ColoniusMazanti}, the continuous-times process 
$t\mapsto \Psi_t$, where $\Psi_t=e^{(t-T_n) A_{L_n}}\Phi_n$ if $t\in[T_n,T_{n+1})$, satisfies almost surely
\be\label{eq:chic}
	\lim_{t\to\infty}\frac{1}{t}\log\left\|\Psi_t\right\|=\chi_d\,\lambda,
\ee
 where we recall that $\lambda$ denotes the coefficient characterizing the exponential distribution $f=f_{U_i}$ as in \eqref{eq:density}. 
  We can define the \emph{Lyapunov exponent of the stochastic process $\Psi_t$}  as the quantity
  \[\chi_\tau(S)=\chi_d\,\lambda.\]

\begin{prop}\label{prop:supportO}
	Let $E=\{1,\dots,m\}$, $m\ge 2$, 
	and $\{A_i \mid i\in E\}\subset M_d(\R)$
	be such that $\{\pi_*A_i\mid i\in E\}$ satisfies \eqref{eq:mainhyp}. 
	Denote by $D\subset \Prd$ the  unique $\tau$-ICS for system \eqref{eq:projsyst} and by 
	$\mu$ the unique invariant ergodic measure  provided by Theorem~\ref{thm:ergodic}.
	For every $i\in E$, let $\mu_i$ be the measure on $\Prd\times E$ defined by 
	$\mu_i(A)=\mu(A\cap (\Prd\times\{i\})$. 
	Then the probability measure $\nu$ on $\Prd\times E$ defined by
		\be\label{defnu}
		\nu =\frac{\lambda^2}{\tau\lambda+1}\sum_{i=1}^m \int_\tau^\infty\left( \int_0^s \left( e^{tA_i},\mathrm{Id} \right)_* \mu_i\, dt\right) e^{-\lambda(s-\tau)}
		ds,
	\ee
where 
$\left( e^{tA_i},\mathrm{Id} \right)_* \mu$ denotes the pushforward measure of $\mu$ along $ (e^{tA_i},\mathrm{Id})$,
satisfies
	\be\label{eq:supportsv}
	\mathrm{supp}\nu\subset \cup_{i\in E,\,t\in [0,\tau]}e^{tA_i}(D)\times\{i\}
	\ee 
 and
	\[
	\chi_\tau(S)=\int_{\Prd\times E}\langle \theta,A_i \theta\rangle d\nu(\theta,i).
	\]
\end{prop}
\begin{proof}
Consider the discrete-time process $Z_n=(Q_n,L_n)$ and denote by $X$ a nonzero 
vector in 
$\R^n$  such that $\pi X=Q_0$. 
For every $t\ge 0$, identify $\Psi_t X$ with $(R_t,Y_t)\in (0,\infty)\times\Prd$ by polar decomposition. The dynamics of $t\mapsto R_t$ is governed by the (stochastic) differential equation
\[
	\frac{\dot{R}_t}{R_t}=\left\langle Y_t,A(t)Y_t \right\rangle,
\]
where $A(t)=A_{L_n}$ if $t\in [T_n,T_{n+1})$. 
Let us recall that $N_t/t$ almost surely converges to  $\frac\lambda{\tau\lambda+1}$ as $t$ tends to infinity. 
We thus have, almost surely, 
\begin{align*}
	\chi_\tau(S)&=
	\lim_{t\to\infty}\frac{1}{t}\log\left(\frac{R_t}{R_0}\right)=
	\frac\lambda{\tau\lambda+1}\lim_{n\to\infty}\frac{1}{n}\log\left(\frac{R_{T_n}}{R_0}\right)\\
	&=\frac\lambda{\tau\lambda+1}\lim_{n\to\infty}\frac{1}{n}\sum_{j=1}^n\log\left(\frac{R_{T_j}}{R_{T_{j-1}}}\right)
	=\frac\lambda{\tau\lambda+1}\lim_{n\to\infty}\frac{1}{n}\sum_{j=1}^n\int_{T_{j-1}}^{T_j}\left\langle Y_t,A(t)Y_t \right\rangle dt\\
	&=\frac\lambda{\tau\lambda+1}\lim_{n\to\infty}\frac{1}{n}\sum_{j=1}^n\int_{0}^{U_j}\left\langle e^{tA_{L_{j-1}}}Y_{T_{j-1}},A_{L_{j-1}} e^{tA_{L_{j-1}}}Y_{T_{j-1}}\right\rangle dt\\
	&=\frac\lambda{\tau\lambda+1}\lim_{n\to\infty}\frac{1}{n}\sum_{j=1}^n\int_{0}^{U_j}\left\langle e^{tA_{L_{j-1}}}Q_{j-1},A_{L_{j-1}} e^{tA_{L_{j-1}}}Q_{j-1}\right\rangle dt.
\end{align*} 
By the Birkhoff ergodic theorem, we have that
\begin{align*}
 \lim_{n\to\infty}\frac{1}{n}&\sum_{j=1}^n\int_{0}^{U_j}\left\langle e^{tA_{L_{j-1}}}Q_{j-1},A_{L_{j-1}} e^{tA_{L_{j-1}}}Q_{j-1}\right\rangle dt\\
&=\lambda \int_{\Prd\times E}\int_0^\infty \int_0^s \left\langle  e^{tA_i}\theta,A_i e^{tA_i}\theta  \right\rangle dte^{-\lambda(s-\tau)}\mathbf{1}_{\{s\ge \tau\}}ds\,d\mu(\theta,i)\\
&=\lambda \int_{\Prd\times E}\int_\tau^\infty \int_0^s \left\langle  e^{tA_i}\theta,A_i e^{tA_i}\theta  \right\rangle dt e^{-\lambda(s-\tau)}ds\,d\mu(\theta,i).
\end{align*}
Hence, 
	\begin{align}
		\chi_\tau(S)&=		\frac{\lambda^2}{\tau\lambda+1}\int_{\Prd\times E}\int_\tau^\infty \int_0^s \left\langle  e^{tA_i}\theta,A_i e^{tA_i}\theta  \right\rangle dte^{-\lambda(s-\tau)}ds\,d\mu(\theta,i)\\
		&=		\frac{\lambda^2}{\tau\lambda+1}\sum_{i=1}^m \int_\tau^\infty \int_0^s \left(\int_{\Prd\times E}\left\langle  \theta,A_i \theta  \right\rangle 
		d((e^{tA_i},\mathrm{Id})_*\mu)(\theta,i)\right)
		dte^{-\lambda(s-\tau)}ds\\
		&=		 \int_{\Prd\times E}\left\langle  \theta,A_i \theta  \right\rangle 
		d\nu(\theta,i),
	\end{align}
	where $\nu$ is defined as in \eqref{defnu}. 

We are left to prove \eqref{eq:supportsv}.
By construction of $\nu$, we have
\[\mathrm{supp}\nu\subset \cup_{i\in E,\,t\ge 0} (e^{tA_i},\mathrm{Id})\mathrm{supp}\mu_i,\]
whence the conclusion, since $\mathrm{supp}\mu_i\subset D\times \{i\}$ and $e^{tA_i}D\subset D$ for every $i\in E$ and $t\ge \tau$.
\end{proof}
	
	\appendix
	
	\section{Existence of a closed orbit}
	
	We prove in this section Theorem~\ref{thm:closedorbit} and we deduce some consequences concerning approximate and exact controllability of projected linear systems. The proof follows the steps suggested to us by Uri Bader and Claudio Procesi, to whom we are very grateful. Every imprecision or naiveness should  be attributed solely to us. Before presenting the proof, we need some preliminary definitions (see also \cite[Chapter 1]{Zimmer}). 
	
	We recall that the set $G/H$ of left cosets of a subgroup $H$ of a topological group $G$ can be endowed with the  quotient topology induced by $\pi:G\to G/H$. If $H$ is a closed subgroup of $G$, then $G/H$ is Hausdorff.
	
	\begin{defi}\label{defi:cocompact}
		Let $G$ be a locally compact topological group. A subgroup $H\subset G$ is \emph{cocompact in $G$} if the quotient space $G/H$ is compact.
	\end{defi}

	\begin{defi}
		A group $G$ is said to be
\emph{solvable} if there exists a finite sequence $\{1\}=G_0\subset G_1\subset\dots\subset G_k=G$ of  subgroups such that $G_{j-1}$ is normal in $G_j$ and $G_j/G_{j-1}$ is abelian for every $j=1,\dots,k$.
	\end{defi}
	
	We also need to recall the following  
	results from the theory of semisimple Lie groups (see, e.g. \cite[\S 3, Chapter VI]{Helgason}).

	\begin{prop}\label{prop:Iwasawa}
		Let $G$ be a noncompact semisimple real algebraic Lie group. Let $K$ denote a maximal compact subgroup
of $G$. Then:
	\begin{itemize}
		\item[i)]There exists an Iwasawa decomposition $G = KAN$, where $A$ is abelian simply
connected (a vector subgroup of $G$) and $N$ is a nilpotent simply connected subgroup of
$G$ preserved by the action of $A$.  
		\item[ii)] Let M be the centralizer of A in K. Then the
subgroup $P_0 = MAN$ is a 
closed cocompact subgroup of $G$, and $AN$ is a closed cocompact connected solvable normal subgroup of $P_0$ (hence a closed cocompact solvable connected subgroup of $G$). 
	\end{itemize}
	\end{prop}

\begin{remark}\label{rem:example}
	An example of an Iwasawa decomposition is given by the special linear group ${\rm SL}(\R,d)=KAN$  with $K={\rm SO}(\R,d)$ the special orthogonal group, $A$ the subgroup of diagonal matrices of ${\rm SL}(\R,d)$, and $N$ the subgroup of lower triangular matrices with all diagonal entries equal to $1$. 
\end{remark}

	\begin{proof}[Proof of Theorem~\ref{thm:closedorbit}]
Let $B$ be a connected Lie subgroup of $\GL(\R,d)$. We should prove that 
the associated representation $\varphi:B\times\mathbb{S}^{d-1}\to \mathbb{S}^{d-1}$ admits at least one closed, hence compact, orbit.
		In fact, it is actually enough to show that there exists a closed cocompact 
subgroup $C$ of $B$ such that $C$ has a  
compact 
orbit $
			\mc{O}_C(x_0)=\left\{ \varphi(c,x_0)
	   \mid c\in C \right\}
		$.  Indeed, assume that such a compact orbit exists and 
		choose a compact subset  $H$ of $B$ such that $B=HC$. 
		Then 
\[
\mc{O}_B(x_0)=\varphi(H\times \mc{O}_C(x_0))
\]
is compact in 	$\mathbb{S}^{d-1}$.

    Now we claim that any connected Lie group has a closed cocompact, connected and solvable 
subgroup, and to prove this assertion we proceed as follows: factoring out the solvable radical,
it is not restrictive to assume $B$ to be a semisimple linear group; in particular $B$ is (real) algebraic. The claim now follows by Proposition~\ref{prop:Iwasawa} above.

We can then		
assume without loss of generality that  $B$ is a connected
		solvable Lie subgroup of $\GL(\R,d)$. 
		By the Lie-Kolchin theorem (see \cite{Kolchin} or, e.g., \cite[Section 17.6]{Hum})
there exists a common eigenvector $v\in \C^d$ for all matrices in $B$. 
The real vector subspace $V=\mathrm{span}\{ v+ \ov{v},\,i(v-\ov{v}) \}$ is then invariant for all matrices in $B$. 
		In the case in which $V$ is one-dimensional,  its projectivization reduces to a singleton, and the proof is complete. If, instead, $V$ has dimension two, then its unit sphere  
		is topologically a circle, and we conclude noticing that the action of a connected group on a circle is either transitive or has a fixed point. 
	Indeed, if the action is not transitive then each orbit is either a point or an open arc, being an homogeneous connected subset of the circle. In the latter case the two endpoints of the arc will be fixed by the action.
	\end{proof}

Theorem~\ref{thm:closedorbit} has an interesting consequence on the controllability properties of the projected system \eqref{eq:projsyst}, which seems new
 even in the case $\tau=0$.

	\begin{defi}\label{defi:contr}
		We say that the projected system \eqref{eq:projsyst} is:
		\begin{itemize}
			\item [i)] \emph{exactly controllable} 
			if, for every $s_0\in\Prd$, the set $\Smt(s_0)$
is the whole $\Prd$, where $\Smt$ is defined as in \eqref{eq:Smt};
			\item [ii)] \emph{approximately controllable}  
			if, for every $s_0\in\Prd$, the set $\Smt(s_0)$
			is dense in $\Prd$.
		\end{itemize} 
	\end{defi}
	
	Then we have the following result.
	
	\begin{proposition}\label{prop:eaxct-approximate}
		The projected system \eqref{eq:projsyst} is exactly controllable if and only if it is approximately controllable.
	\end{proposition}
	
	\begin{proof}
		One implication being trivial, we should just prove that the approximate controllability of \eqref{eq:projsyst} implies its exact controllability. By Theorem~\ref{thm:closedorbit}, let us fix $s_0\in\Prd$ such that the orbit $\mc{O}(s_0)$ is closed. Then  the  chain of inclusions
		\[
			\Smt(s_0)\subset\mc{O}(s_0)\subset \Prd
		\]
		implies, upon passing to the closures, that $\mc{O}(s_0)=\Prd$. 
		By point i) of Proposition~\ref{prop:orbitanalytic} we thus deduce that 
		 \eqref{eq:projsyst} satisfies hypothesis \eqref{eq:mainhyp} on $\Prd$.
		
		Take now $s_0,s_1\in \Prd$ and let us prove that $s_1\in\Smt(s_0)$. Applying Proposition~\ref{prop:notempty} to the time-reverse system, we deduce that there exists 
		a nonempty open subset $\Omega$ of  $\Prd$ such that 
		\[s_1\in \Smt(\hat s)\qquad \forall\hat s\in \Omega.\]
By the approximate controllability assumption, moreover, $\Smt(s_0)\cap \Omega\ne \emptyset$. By concatenating a signal in $\mc{S}^\tau$ driving $s_0$ to some $\hat s\in \Omega$ and another signal in $\mc{S}^\tau$ driving $\hat s$ to $s_1$, we get the desired conclusion. 
	\end{proof}
	
	\begin{remark}
		For a general nonlinear system
		approximate and exact controllability are not equivalent, unless  the 
		Lie algebra rank condition
		is assumed to hold. 
		Here, instead, no 
		Lie algebra rank condition 
is assumed and  the result is a consequence of the special structure of the system (which guarantees the existence of a closed orbit, as stated in Theorem~\ref{thm:closedorbit}). 
		Proposition~\ref{prop:eaxct-approximate} extends the equivalence between approximate and exact controllability obtained in \cite[Theorem 17]{SBRG}  for closed finite-dimensional quantum systems. Using the notations of the present paper, the class of systems studied in \cite[Theorem 17]{SBRG} corresponds to the case where $\tau=0$, $d$ is even, and $S$ is contained in $U(d/2)$ (seen as a subset of $M_{d}(\R)$, up to the canonical identification of $\C^{d/2}$ and $\R^{d}$). 
\end{remark}

\begin{remark}
By \cite[Propositions 1 and 2]{BacciottiVivalda}, 
($\pi\Sigma_0$)  is exactly controllable on $\Prd$ if and only if the projected system on the sphere $\mathbb{S}^{d-1}$
is exactly controllable. 
In particular, Proposition~\ref{prop:eaxct-approximate} extends to the projection of 
systems of the type ($\Sigma_0$)
on the sphere.
\end{remark}

	\bibliographystyle{abbrv}
	\bibliography{Biblio}
	
\end{document}